\documentclass[10pt]{amsart}

\usepackage[margin=1in,headheight=13.6pt]{geometry}

\usepackage{textcmds} 
\usepackage{amsmath}
\usepackage{amsxtra}
\usepackage{amscd}
\usepackage{amsthm}
\usepackage{amsfonts}
\usepackage{amssymb}
\usepackage{eucal}
\usepackage[all]{xy}
\usepackage{graphicx}
\usepackage{comment}
\usepackage{epsfig}
\usepackage{psfrag}
\usepackage{mathrsfs}
\usepackage{amscd}
\usepackage{rotating}
\usepackage{lscape}
\usepackage{amsbsy}
\usepackage{verbatim}
\usepackage{moreverb}
\usepackage{url}
\usepackage{extarrows}
\usepackage{setspace}

\usepackage{cancel}
\usepackage{bbm}

\usepackage[hidelinks]{hyperref}

\usepackage{tikz-cd}
\usepackage{tikz}
\usetikzlibrary{matrix,arrows,decorations.pathmorphing}

\DeclareMathAlphabet{\mathpzc}{OT1}{pzc}{m}{it}
\newtheorem{thm}[subsubsection]{Theorem}

\newtheorem{mainthm}{Theorem}

\newenvironment{mainthmbis}[1]
  {\renewcommand{\themainthm}{\ref{#1}$'$}%
   \addtocounter{mainthm}{-1}%
   \begin{mainthm}}
  {\end{mainthm}}
  
  \newtheorem{mainconj}[mainthm]{Conjecture}
  
  \newenvironment{mainconjbis}[1]
  {\renewcommand{\themainthm}{\ref{#1}$'$}%
   \addtocounter{mainthm}{-1}%
   \begin{mainconj}}
  {\end{mainconj}}

\newtheorem{cor}[subsubsection]{Corollary}
\newtheorem{lem}[subsubsection]{Lemma}
\newtheorem{prop}[subsubsection]{Proposition}

\newtheorem{defn}[subsubsection]{Definition}

\theoremstyle{remark}
\newtheorem{rem}[subsubsection]{Remark}
\newtheorem{example}[subsubsection]{Example}

\numberwithin{equation}{section}

\newcommand{\nc}{\newcommand}
\nc{\renc}{\renewcommand}
\nc{\ssec}{\subsection}
\nc{\sssec}{\subsubsection}
\nc{\on}{\operatorname}

\nc\ol{\overline}
\nc\wt{\widetilde}
\nc\tboxtimes{\wt{\boxtimes}}
\nc\tstar{\wt{\star}}
\nc{\alp}{\alpha}

\nc{\ZZ}{{\mathbb Z}}
\nc{\NN}{{\mathbb N}}
\nc{\OO}{{\mathbb O}}
\renc{\SS}{{\mathbb S}}
\nc{\DD}{{\mathbb D}}
\nc{\GG}{{\mathbb G}}
\renewcommand{\AA}{{\mathbb A}}
\nc{\Fq}{{\mathbb F}_q}
\nc{\Fqb}{\ol{{\mathbb F}_q}}
\nc{\Ql}{\ol{{\mathbb Q}_\ell}}
\nc{\id}{\on{id}}
\nc\X{\mathcal X}

\nc{\Hom}{\on{Hom}}
\nc{\Lie}{\on{Lie}}
\nc{\Loc}{\on{Loc}}
\nc{\Pic}{\on{Pic}}
\nc{\Bun}{\on{Bun}}
\nc{\IC}{\on{IC}}
\nc{\Aut}{\on{Aut}}
\nc{\rk}{\on{rk}}
\nc{\Sh}{\on{Sh}}
\nc{\Perv}{\on{Perv}}
\nc{\pos}{{\on{pos}}}
\nc{\Conv}{\on{Conv}}
\nc{\Sph}{\on{Sph}}
\nc{\Sym}{\on{Sym}}
\nc{\BunBb}{\overline{\Bun}_B}
\nc{\BunNb}{\overline{\Bun}_N}
\nc{\BunTb}{\overline{\Bun}_T}
\nc{\BunBbm}{\overline{\Bun}_{B^-}}
\nc{\BunBbel}{\overline{\Bun}_{B,el}}
\nc{\BunBbmel}{\overline{\Bun}_{B^-,el}}
\nc{\Buno}{\overset{o}{\Bun}}
\nc{\BunPb}{{\overline{\Bun}_P}}
\nc{\BunBM}{\Bun_{B(M)}}
\nc{\BunBMb}{\overline{\Bun}_{B(M)}}
\nc{\BunPbw}{{\widetilde{\Bun}_P}}
\nc{\BunBP}{\widetilde{\Bun}_{B,P}}
\nc{\GUb}{\overline{G/U}}
\nc{\GUPb}{\overline{G/U(P)}}

\nc\syminfty{\on{Sym}^{\infty}}
\nc\lal{\ol{\lambda}}
\nc\xl{\ol{x}}
\nc\thl{\ol{\theta}}
\nc\nul{\ol{\nu}}
\nc\mul{\ol{\mu}}
\nc\Sum\Sigma
\nc{\oX}{\overset{\circ}{X}{}}
\nc{\hl}{\overset{\leftarrow}h{}}
\nc{\hr}{\overset{\rightarrow}h{}}
\nc{\M}{{\mathcal M}}
\nc{\N}{{\mathcal N}}
\nc{\F}{{\mathcal F}}
\nc{\D}{{\mathcal D}}
\nc{\Y}{{\mathcal Y}}
\nc{\G}{{\mathcal G}}
\nc{\E}{{\mathcal E}}
\nc{\CalC}{{\mathcal C}}
\nc\Dh{\widehat{\D}}
\renewcommand{\O}{{\mathcal O}}
\nc{\K}{{\mathcal K}}
\renewcommand{\S}{{\mathcal S}}
\nc{\T}{{\mathcal T}}
\nc{\V}{{\mathcal V}}
\renc{\P}{{\mathcal P}}
\nc{\A}{{\AA}}
\nc{\B}{{\BB}}
\nc{\U}{{\mathcal U}}
\renewcommand{\L}{{\mathcal L}}

\nc{\frn}{{\check{\mathfrak u}(P)}}

\nc{\fC}{\mathfrak C}
\nc\f{{\mathfrak f}}

\nc{\qo}{{\mathfrak q}}
\nc{\po}{{\mathfrak p}}
\nc{\s}{{\mathfrak s}}
\nc\w{\text{w}}
\renewcommand{\r}{{\mathfrak r}}

\nc\mathi\iota
\nc\Spec{\on{Spec}}
\nc\Mod{\on{Mod}}
\nc{\tw}{\widetilde{\mathfrak t}}
\nc{\pw}{\widetilde{\mathfrak p}}
\nc{\qw}{\widetilde{\mathfrak q}}
\nc{\jw}{\widetilde j}

\nc{\grb}{\overline{\Gr_{X^{\fset}}}}
\nc{\I}{\mathcal I}
\renewcommand{\i}{\mathfrak i}
\renewcommand{\j}{\mathfrak j}

\nc{\lambdach}{{\check\lambda}}
\nc{\Lambdach}{{\check\Lambda}{}}
\nc{\much}{{\check\mu}}
\nc{\omegach}{{\check\omega}}
\nc{\nuch}{{\check\nu}}
\nc{\etach}{{\check\eta}}
\nc{\alphach}{{\check\alpha}}

\nc{\rhoch}{{\check\rho}}

\nc{\Hb}{\overline{\H}}

\nc{\BA}{{\mathbb{A}}}
\nc{\BB}{\mathbb{B}}
\nc{\BC}{{\mathbb{C}}}
\nc{\BD}{{\mathbb{D}}}
\nc{\BE}{{\mathbb{E}}}
\nc{\BF}{{\mathbb{F}}}
\nc{\BG}{{\mathbb{G}}}
\nc{\BH}{{\mathbb{H}}}
\nc{\BI}{{\mathbb{I}}}
\nc{\BM}{{\mathbb{M}}}
\nc{\BN}{{\mathbb{N}}}
\nc{\BO}{{\mathbb{O}}}
\nc{\BP}{{\mathbb{P}}}
\nc{\BQ}{{\mathbb{Q}}}
\nc{\BR}{{\mathbb{R}}}
\nc{\BS}{{\mathbb{S}}}
\nc{\BT}{{\mathbb{T}}}
\nc{\BV}{{\mathbb{V}}}
\nc{\BZ}{{\mathbb{Z}}}

\nc{\bbone}{\mathbbm{1}}
\nc{\bbA}{{\mathbb{A}}}
\nc{\bbB}{\mathbb{B}}
\nc{\bbC}{{\mathbb{C}}}
\nc{\bbD}{{\mathbb{D}}}
\nc{\bbE}{{\mathbb{E}}}
\nc{\bbF}{{\mathbb{F}}}
\nc{\bbG}{{\mathbb{G}}}
\nc{\bbH}{{\mathbb{H}}}
\nc{\bbI}{{\mathbb{I}}}
\nc{\bbL}{{\mathbb{L}}}
\nc{\bbM}{{\mathbb{M}}}
\nc{\bbN}{{\mathbb{N}}}
\nc{\bbO}{{\mathbb{O}}}
\nc{\bbP}{{\mathbb{P}}}
\nc{\bbQ}{{\mathbb{Q}}}
\nc{\bbR}{{\mathbb{R}}}
\nc{\bbS}{{\mathbb{S}}}
\nc{\bbT}{{\mathbb{T}}}
\nc{\bbU}{{\mathbb{U}}}
\nc{\bbV}{{\mathbb{V}}}
\nc{\bbW}{{\mathbb{W}}}
\nc{\bbX}{{\mathbb{X}}}
\nc{\bbY}{{\mathbb{Y}}}
\nc{\bbZ}{{\mathbb{Z}}}

\nc{\CA}{{\mathcal{A}}}
\nc{\CB}{{\mathcal{B}}}

\nc{\CE}{{\mathcal{E}}}
\nc{\CF}{{\mathcal{F}}}
\nc{\CH}{{\mathcal{H}}}

\nc{\CL}{{\mathcal{L}}}
\nc{\CC}{{\mathcal{C}}}
\nc{\CG}{{\mathcal{G}}}
\nc{\CM}{{\mathcal{M}}}
\nc{\CN}{{\mathcal{N}}}
\nc{\CK}{{\mathcal{K}}}
\nc{\CO}{{\mathcal{O}}}
\nc{\CP}{{\mathcal{P}}}
\nc{\CQ}{{\mathcal{Q}}}
\nc{\CR}{{\mathcal{R}}}
\nc{\CS}{{\mathcal{S}}}
\nc{\CU}{{\mathcal{U}}}
\nc{\CV}{{\mathcal{V}}}
\nc{\CW}{{\mathcal{W}}}
\nc{\CX}{{\mathcal{X}}}
\nc{\CY}{{\mathcal{Y}}}
\nc{\CZ}{{\mathcal{Z}}}
\nc{\CI}{{\mathcal{I}}}

\nc{\csM}{{\check{\mathcal A}}{}}
\nc{\oM}{{\overset{\circ}{\mathcal M}}{}}
\nc{\obM}{{\overset{\circ}{\mathbf M}}{}}
\nc{\oCA}{{\overset{\circ}{\mathcal A}}{}}
\nc{\obA}{{\overset{\circ}{\mathbf A}}{}}
\nc{\ooM}{{\overset{\circ}{M}}{}}
\nc{\osM}{{\overset{\circ}{\mathsf M}}{}}
\nc{\vM}{{\overset{\bullet}{\mathcal M}}{}}
\nc{\nM}{{\underset{\bullet}{\mathcal M}}{}}
\nc{\oD}{{\overset{\circ}{\mathcal D}}{}}
\nc{\obC}{{\overset{\circ}{\mathbf C}}{}}
\nc{\obD}{{\overset{\circ}{\mathbf D}}{}}
\nc{\oA}{{\overset{\circ}{\mathbb A}}{}}
\nc{\op}{{\overset{\bullet}{\mathbf p}}{}}
\nc{\oU}{{\overset{\bullet}{\mathcal U}}{}}
\nc{\oZ}{{\overset{\circ}{\mathcal Z}}{}}
\nc{\ofZ}{{\overset{\circ}{\mathfrak Z}}{}}
\nc{\oF}{{\overset{\circ}{\fF}}}

\nc{\fa}{{\mathfrak{a}}}
\nc{\fb}{{\mathfrak{b}}}
\nc{\fc}{{\mathfrak{c}}}
\nc{\fd}{{\mathfrak{d}}}
\nc{\ff}{{\mathfrak{f}}}
\nc{\fg}{{\mathfrak{g}}}
\nc{\fgl}{{\mathfrak{gl}}}
\nc{\fh}{{\mathfrak{h}}}
\nc{\fj}{{\mathfrak{j}}}
\nc{\fl}{{\mathfrak{l}}}
\nc{\fm}{{\mathfrak{m}}}
\nc{\fn}{{\mathfrak{n}}}
\nc{\fu}{{\mathfrak{u}}}
\nc{\fp}{{\mathfrak{p}}}
\nc{\fr}{{\mathfrak{r}}}
\nc{\fs}{{\mathfrak{s}}}
\nc{\ft}{{\mathfrak{t}}}
\nc{\fz}{{\mathfrak{z}}}
\nc{\fsl}{{\mathfrak{sl}}}
\nc{\hsl}{{\widehat{\mathfrak{sl}}}}
\nc{\hgl}{{\widehat{\mathfrak{gl}}}}
\nc{\hg}{{\widehat{\mathfrak{g}}}}
\nc{\chg}{{\widehat{\mathfrak{g}}}{}^\vee}
\nc{\hn}{{\widehat{\mathfrak{n}}}}
\nc{\chn}{{\widehat{\mathfrak{n}}}{}^\vee}

\nc{\fA}{{\mathfrak{A}}}
\nc{\fB}{{\mathfrak{B}}}
\nc{\fD}{{\mathfrak{D}}}
\nc{\fE}{{\mathfrak{E}}}
\nc{\fF}{{\mathfrak{F}}}
\nc{\fG}{{\mathfrak{G}}}
\nc{\fK}{{\mathfrak{K}}}
\nc{\fL}{{\mathfrak{L}}}
\nc{\fM}{{\mathfrak{M}}}
\nc{\fN}{{\mathfrak{N}}}
\nc{\fP}{{\mathfrak{P}}}
\nc{\fU}{{\mathfrak{U}}}
\nc{\fV}{{\mathfrak{V}}}
\nc{\fX}{{\mathfrak{X}}}
\nc{\fY}{{\mathfrak{Y}}}
\nc{\fZ}{{\mathfrak{Z}}}

\nc{\bb}{{\mathbf{b}}}
\nc{\bc}{{\mathbf{c}}}
\nc{\bd}{{\mathbf{d}}}
\nc{\bbf}{{\mathbf{f}}}
\nc{\be}{{\mathbf{e}}}
\nc{\bg}{{\mathbf{g}}}
\nc{\bi}{{\mathbf{i}}}
\nc{\bj}{{\mathbf{j}}}
\nc{\bn}{{\mathbf{n}}}
\nc{\bo}{{\mathbf{o}}}
\nc{\bp}{{\mathbf{p}}}
\nc{\bq}{{\mathbf{q}}}
\nc{\bt}{{\mathbf{t}}}
\nc{\bu}{{\mathbf{u}}}
\nc{\bv}{{\mathbf{v}}}
\nc{\bx}{{\mathbf{x}}}
\nc{\bs}{{\mathbf{s}}}
\nc{\by}{{\mathbf{y}}}
\nc{\bw}{{\mathbf{w}}}
\nc{\bA}{{\mathbf{A}}}
\nc{\bK}{{\mathbf{K}}}
\nc{\bB}{{\mathbf{B}}}
\nc{\bC}{{\mathbf{C}}}
\nc{\bG}{{\mathbf{G}}}
\nc{\bD}{{\mathbf{D}}}
\nc{\bH}{{\mathbf{H}}}
\nc{\bM}{{\mathbf{M}}}
\nc{\bN}{{\mathbf{N}}}
\nc{\bO}{{\mathbf{O}}}
\nc{\bT}{{\mathbf{T}}}
\nc{\bV}{{\mathbf{V}}}
\nc{\bW}{{\mathbf{W}}}
\nc{\bX}{{\mathbf{X}}}
\nc{\bZ}{{\mathbf{Z}}}
\nc{\bS}{{\mathbf{S}}}

\nc{\sA}{{\mathsf{A}}}
\nc{\sB}{{\mathsf{B}}}
\nc{\sC}{{\mathsf{C}}}
\nc{\sD}{{\mathsf{D}}}
\nc{\sF}{{\mathsf{F}}}
\nc{\sG}{{\mathsf{G}}}
\nc{\sK}{{\mathsf{K}}}
\nc{\sM}{{\mathsf{M}}}
\nc{\sO}{{\mathsf{O}}}
\nc{\sW}{{\mathsf{W}}}
\nc{\sQ}{{\mathsf{Q}}}
\nc{\sP}{{\mathsf{P}}}
\nc{\sV}{{\mathsf{V}}}
\nc{\sS}{{\mathsf{S}}}
\nc{\sT}{{\mathsf{T}}}
\nc{\sZ}{{\mathsf{Z}}}
\nc{\sfp}{{\mathsf{p}}}
\nc{\sll}{{\mathsf{l}}}
\nc{\sr}{{\mathsf{r}}}
\nc{\bk}{{\mathsf{k}}}
\nc{\sg}{{\mathsf{g}}}

\nc{\sff}{{\mathsf{f}}}
\nc{\sfb}{{\mathsf{b}}}
\nc{\sfc}{{\mathsf{c}}}
\nc{\sd}{{\mathsf{d}}}
\nc{\se}{{\mathsf{e}}}

\nc{\BK}{{\bar{K}}}

\nc{\tA}{{\widetilde{\mathbf{A}}}}
\nc{\tB}{{\widetilde{\mathcal{B}}}}
\nc{\tg}{{\widetilde{\mathfrak{g}}}}
\nc{\tG}{{\widetilde{G}}}
\nc{\TM}{{\widetilde{\mathbb{M}}}{}}
\nc{\tO}{{\widetilde{\mathsf{O}}}{}}
\nc{\tU}{{\widetilde{\mathfrak{U}}}{}}
\nc{\TZ}{{\tilde{Z}}}
\nc{\tx}{{\tilde{x}}}
\nc{\tbv}{{\tilde{\bv}}}
\nc{\tfP}{{\widetilde{\mathfrak{P}}}{}}
\nc{\tz}{{\tilde{\zeta}}}
\nc{\tmu}{{\tilde{\mu}}}

\nc{\urho}{\underline{\rho}}
\nc{\uB}{\underline{B}}
\nc{\uC}{{\underline{\mathbb{C}}}}
\nc{\ui}{\underline{i}}
\nc{\uj}{\underline{j}}
\nc{\ofP}{{\overline{\mathfrak{P}}}}
\nc{\oB}{{\overline{\mathcal{B}}}}
\nc{\og}{{\overline{\mathfrak{g}}}}
\nc{\oI}{{\overline{I}}}

\nc{\eps}{\varepsilon}
\nc{\hrho}{{\hat{\rho}}}

\nc{\one}{{\mathbf{1}}}
\nc{\two}{{\mathbf{t}}}

\nc{\Rep}{{\mathop{\operatorname{\rm Rep}}}}
\nc{\Tot}{{\mathop{\operatorname{\rm Tot}}}}
\nc{\Ker}{{\mathop{\operatorname{\rm Ker}}}}
\nc{\Hilb}{{\mathop{\operatorname{\rm Hilb}}}}
\nc{\Ext}{{\mathop{\operatorname{\rm Ext}}}}
\nc{\CHom}{{\mathop{\operatorname{{\mathcal{H}}\it om}}}}
\nc{\GL}{{\mathop{\operatorname{\rm GL}}}}
\nc{\gr}{{\mathop{\operatorname{\rm gr}}}}
\nc{\Id}{{\mathop{\operatorname{\rm Id}}}}
\nc{\de}{{\mathop{\operatorname{\rm def}}}}
\nc{\length}{{\mathop{\operatorname{\rm length}}}}
\nc{\supp}{{\mathop{\operatorname{\rm supp}}}}

\nc{\Cliff}{{\mathsf{Cliff}}}
\nc{\Fl}{\on{Fl}}
\nc{\Fib}{{\mathsf{Fib}}}
\nc{\Coh}{{\on{Coh}}}
\nc{\QCoh}{{\on{QCoh}}}
\nc{\IndCoh}{{\on{IndCoh}}}
\nc{\FCoh}{{\mathsf{FCoh}}}

\nc{\reg}{{\text{\rm reg}}}

\nc{\cplus}{{\mathbf{C}_+}}
\nc{\cminus}{{\mathbf{C}_-}}
\nc{\cthree}{{\mathbf{C}_*}}
\nc{\Qbar}{{\bar{Q}}}
\nc\Eis{\on{Eis}}
\nc\CT{\on{CT}}
\nc\Eisb{\ol\Eis{}}
\nc\Eisr{\on{Eis}^{rat}{}}
\nc\wh{\widehat}
\nc{\Def}{\on{Def_{\check{\fb}}(E)}}
\nc{\barZ}{\overline{Z}{}}
\nc{\barbarZ}{\overline{\barZ}{}}
\nc{\barpi}{\overline\pi}
\nc{\barbarpi}{\overline\barpi}
\nc{\barpip}{\overline\pi{}^+}
\nc{\barpim}{\overline\pi{}^-}

\nc{\fq}{\mathfrak q}

\nc{\fqb}{\ol{\fq}{}}
\nc{\fpb}{\ol{\fp}{}}
\nc{\fpr}{{\fp^{rat}}{}}
\nc{\fqr}{{\fq^{rat}}{}}

\nc{\hattimes}{\wh\otimes}

\nc{\bh}{{\bar{h}}}
\nc{\bOmega}{{\overline{\Omega(\check \fn)}}}

\nc{\seq}[1]{\stackrel{#1}{\sim}}

\nc{\cT}{{\check{T}}}
\nc{\cG}{{\check{G}}}
\nc{\cM}{{\check{M}}}
\nc{\cB}{{\check{B}}}
\nc{\cP}{{\check{P}}}

\nc{\ct}{{\check{\mathfrak t}}}
\nc{\cg}{{\check{\fg}}}
\nc{\cb}{{\check{\fb}}}
\nc{\cn}{{\check{\fn}}}
\nc{\cp}{{\check{\fp}}}
\nc{\cm}{{\check{\fm}}}

\nc{\cLambda}{{\check\Lambda}}

\nc{\cla}{{\check\lambda}}
\nc{\cmu}{{\check\mu}}
\nc{\cnu}{{\check\nu}}
\nc{\ceta}{{\check\eta}}

\nc{\DefbE}{{\on{Def}_{\cB}(E_\cT)}}

\nc{\imathb}{{\ol{\imath}}}
\nc{\rlr}{\overset{\longrightarrow}{\underset{\longrightarrow}\longleftarrow}}

\nc{\oBun}{\overset{\circ}\Bun}
\nc{\BunBbb}{\ol{\ol{Bun}}_B}
\nc{\BunBr}{\Bun_B^{rat}}
\nc{\BunBrsg}{\Bun_B^{rat,\on{s.g.}}}
\nc{\BunBrp}{\Bun_B^{rat,polar}}
\nc{\BunBrpbg}{\Bun_B^{rat,polar,\on{b.g.}}}
\nc{\BunBrpsg}{\Bun_B^{rat,polar,\on{s.g.}}}
\nc{\BunTrp}{\Bun_T^{rat,polar}}
\nc{\BunTrpbg}{\Bun_T^{rat,polar,\on{b.g.}}}
\nc{\BunTrpsg}{\Bun_T^{rat,polar,\on{s.g.}}}
\nc{\BunNr}{\Bun_N^{rat}}
\nc{\BunNre}{\Bun_N^{enh,rat}}
\nc{\BunTr}{\Bun_T^{rat}}
\nc{\Vect}{\on{Vect}}
\nc{\Whit}{\on{Whit}}
\nc{\bTb}{\ol{\on{CT}}}

\nc{\bTr}{\on{CT}^{rat}{}}
\nc\jmathr{\jmath^{rat}{}}
\nc{\ux}{\underline{x}}
\nc{\clambda}{{\check\lambda}}
\nc{\calpha}{{\check\alpha}}

\nc{\inftyGrpd}{{\mathsf{Grpd}_\infty}}

\nc{\fset}{\mathsf{fSet}}
\nc{\LocSysG}{\LocSys_{\cG}}
\nc{\Sing}{{\on{Sing}}}
\nc{\dr}{{\on{dR}}}
\nc{\Ind}{\on{Ind}}
\nc{\Sat}{\on{Sat}}
\nc{\Ho}{\on{Ho}}
\nc{\Res}{\on{Res}}
\nc{\sotimes}{\overset{!}\otimes}
\nc{\mmod}{{\on{-}}{\mathbf{mod}}}
\nc{\Maps}{\on{Maps}}
\nc{\CMaps}{{\mathcal Maps}}
\nc{\bMaps}{{\mathbf{Maps}}}

\nc{\dgSch}{\on{DGSch}}
\nc{\dgindSch}{\on{DGindSch}}
\nc{\indSch}{\on{indSch}}
\nc{\Sch}{\mathsf{Sch}}
\nc{\affdgSch}{\on{DGSch}^{\on{aff}}}
\nc{\affSch}{\on{Sch}^{\on{aff}}}

\nc{\Groupoids}{\on{Grpd}}
\nc{\inftypic}{\infty\on{-PicGrpd}}
\nc{\inftyCat}{{\mathsf{Cat}_{\infty}}}
\nc{\MoninftyCat}{\infty\on{-Cat}^{Mon}}
\nc{\SymMoninftyCat}{\infty\on{-Cat}^{\on{SymMon}}}
\nc{\SymMonStinftyCat}{\on{DGCat}^{\on{SymMon}}}
\nc{\MonStinftyCat}{\on{DGCat}^{Mon}}
\nc{\inftystack}{\on{Stk}}
\nc{\inftystackalg}{Stk^{1\text{-}alg}}
\nc{\inftyprestack}{\on{PreStk}}
\nc{\inftydgnearstack}{\on{NearStk}}
\nc{\inftydgstack}{\on{Stk}}
\nc{\inftydgstackalg}{DGStk^{1\text{-}alg}}
\nc{\inftydgprestack}{\on{PreStk}}

\nc{\HC}{\CH\bC}

\nc{\csupp}{\supp}
\nc{\Arth}{\on{Arth}}
\nc{\ArthG}{{\on{Arth}_\cG}}
\nc{\ul}{\underline}

\nc{\Z}{\mathcal{Z}}

\nc{\calN}{\N}
\nc{\calW}{\mathcal{W}}
\nc{\calF}{\mathcal{F}}
\nc{\calH}{\mathcal{H}}
\nc{\calO}{\mathcal{O}}
\nc{\calK}{\mathcal{K}}

\nc{\Ran}{\mathsf{Ran}}
\nc{\Jets}{\on{Jets}}
\nc{\act}{\mathsf{act}}
\nc{\Av}{\mathsf{Av}}
\nc{\Ad}{\on{Ad}}
\nc{\BGRan}{BG_{\Ran}}
\nc{\colim}{\on{colim}}
\nc{\codim}{\on{codim}}
\nc{\cpt}{{\on{cpt}}}
\nc{\dR}{{\on{dR}}}
\nc{\DGCat}{\mathsf{DGCat}}
\nc{\DGCatcont}{\on{DGCat}_{cont}}
\nc{\glob}{{\on{glob}}}
\nc{\loc}{{\on{loc}}}

\renewcommand{\op}{{\on{op}}}
\nc{\pt}{{\on{pt}}}
\nc{\PreStk}{{\mathsf{PreStk}}}
\nc{\Cat}{{\mathsf{Cat}}}

\nc{\ShvCat}{{\mathsf{ShvCat}}}

\nc{\restr}[2]{\left. #1 \right |_{#2}}
\nc{\uprestr}[2]{\left. #1 \right |^{#2}}

\nc{\bLoc}{{\mathbf{Loc}}}
\nc{\bGamma}{{\mathbf{\Gamma}}}
\nc{\bLocA}{\mathbf{Loc}^\A}
\nc{\bGammaA}{\mathbf{\Gamma}^\A}
\nc{\bLocB}{\mathbf{Loc}^\B}
\nc{\bGammaB}{\mathbf{\Gamma}^\B}
\nc{\bLocH}{\mathbf{Loc}^\H}
\nc{\bGammaH}{\mathbf{\Gamma}^\H}
\nc{\gen}{\mathsf{gen}}

\nc{\hto}{\hookrightarrow}

\nc{\ext}{\mathsf{ext}}

\nc{\ev}{\mathsf{ev}}
\nc{\rat}{\mathsf{rat}}

\nc{\usotimes}[1]{\underset{#1}{\otimes}}
\nc{\ustimes}[1]{\underset{#1}{\times}}
\nc{\uscolim}[1]{\underset{#1}{\colim}}

\nc{\ch}{{\mathfrak{ch}}}

\renc{\fD}{{\Dmod}}
\nc{\fH}{{\mathfrak{H}}}

\nc{\p}{{\mathfrak{p}}}
\renc{\r}{{\mathfrak{r}}}

\nc{\xto}{\xrightarrow}

\renc{\sec}{\section}
\nc{\enh}{{\on{enh}}}

\renc{\gen}{\mathsf{gen}}
\nc{\BunGBgen}{\Bun_G^{B-\gen}}
\nc{\BunGHgen}{\Bun_G^{H-\gen}}
\nc{\BunGNgen}{\Bun_G^{N-\gen}}

\nc{\Fun}{\mathsf{Fun}}
\nc{\End}{\mathsf{End}}

\nc{\lr}{\xymatrix{ \ar@<-0.4ex>[r] \ar@<.5ex>[l]  & } }
\nc{\rr}{\xymatrix{ \ar@<-0.2ex>[r] \ar@<.7ex>[r]  & } }
\nc{\rrr}{\xymatrix{ \ar@<.0ex>[r] \ar@<.7ex>[r] \ar@<-0.7ex>[r] & } }

\nc{\Stab}{\mathsf{Stab}}
\nc{\Orb}{\mathsf{Orb}}

\renc{\exp}{\mathit{exp}}

\renc{\q}{\mathfrak{q}}

\nc{\virg}[1]{``#1"}

\nc{\QA}[2]
{\textbf{Question:} {#1} 
\\
\textbf{Answer:} {#2}}

\renc{\bold}[1]{\boldsymbol{#1}}

\nc{\bigt}[1]{\big( #1 \big) }
\nc{\Bigt}[1]{\Big( #1 \Big) }

\nc{\extwhit}{{\CW h}(G,\mathsf{ext})}

\nc{\footcite}{\footnote}

\nc{\GA}{{G(\AA)}}
\nc{\GO}{{G(\OO)}}

\nc{\Shv}{\mathsf{Shv}}
\nc{\inc}{\mathsf{inc}}

\nc{\Par}{\mathsf{Par}}

\renc{\i}{\mathfrak{i}}

\nc{\NA}{N(\AA)}
\nc{\VA}{V(\AA)}

\nc{\Glue}{\mathsf{Glue}}
\nc{\laxlim}{\text{laxlim}}

\nc{\FT}{\mathsf{FT}}

\nc{\out}{\mathsf{out}}

\nc{\hol}{\mathsf{hol}}
\nc{\Hol}{\on{Hol}}
\nc{\add}{\mathsf{add}}

\nc{\sto}{\rightsquigarrow}
\nc{\squigto}{\rightsquigarrow}

\nc{\fW}{\mathfrak{W}}

\nc{\vrho}{\varrho}

\nc{\counit}{\mathsf{counit}}
\nc{\unit}{\mathsf{unit}}
\nc{\corr}{\mathsf{corr}}
\nc{\Corr}{\mathsf{Corr}}

\nc{\IndSch}{\mathsf{IndSch}}
\nc{\Tate}{{\mathsf{Tate}}}

\nc{\surjto}{\twoheadrightarrow}

\renc{\j}{\mathfrak{j}}

\nc{\J}{\mathcal{J}}

\nc{\pro}{\mathsf{pro}}
\nc{\fty}{\mathsf{ft}}
\nc{\Pro}{\mathsf{Pro}}

\nc{\coact}{\mathsf{coact}}
\nc{\aff}{\mathsf{aff}}

\nc{\Nilp}{\on{Nilp}}

\nc{\Gch}{{\check{G}}}
\nc{\Pch}{{\check{P}}}
\nc{\Mch}{{\check{M}}}
\nc{\Qch}{{\check{Q}}}

\nc{\LL}{\mathbb{L}}

\nc{\LS}{{\on{LS}}}

\nc{\x}{\varkappa} %%to denote a point of Ran(X).

\nc{\Otimes}{\boldsymbol{\otimes}}
\nc{\Times}{\boldsymbol{\times}}
\nc{\flip}{\text{<}}

\nc{\coeffRan}{\mathsf{coeff}^{\Ran}}

\nc{\Ha}{H(\sA)}
\nc{\Groups}{\mathsf{Groups}}

\nc{\Groth}{\mathsf{Groth}}

\nc{\rlto}{\rightleftarrows}

\nc{\DGCatRan}{\ShvCatCrys(\Ran)}

\nc{\longto}{\longrightarrow}

\renc{\Jets}{\mathsf{Jets}}
\nc{\mer}{\mathsf{mer}}

\nc{\W}{\mathcal{W}}

\nc{\Sect}{\mathsf{Sect}}
\renc{\Maps}{\mathsf{Maps}}

\renc{\bf}{\mathbf{f}}

\nc{\y}{\mathtt{y}}
\renc{\x}{\mathtt{x}}

\nc{\un}{{\it un}}
\nc{\indep}{\mathsf{indep}}
\nc{\CoAlg}{\mathsf{CoAlg}}

\nc{\coeff}{\mathsf{coeff}}

\nc{\R}{\mathcal{R}}

\renc{\hat}{\widehat}

\nc{\TK}{T(\mathsf{K})} %%T(K) Ran with no twist
\nc{\TtKK}{\Tt(\mathpzc{K})} %%independent T(K) with twist
\nc{\TtK}{\Tt(\mathsf{K})} %%T(K) Ran with twist
\nc{\KK}{\mathpzc{K}}

\nc{\Dmod}{\mathfrak{D}}

\nc{\curs}[1]{\mathpzc{#1}}

\nc{\Bshv}{\bold{\B}}
\nc{\Bind}{\H_{\indep}}
\nc{\BRan}{\H_{\Ran}}
\nc{\ARan}{\A_{\Ran}}
\nc{\Aind}{\A_{\indep}}

\nc{\GrRan}{\Gr}
\nc{\Gr}{\mathsf{Gr}}
\nc{\GrGRan}{\Gr_{G}}
\nc{\GrGind}{\Gr_{G}^{\indep}}
\nc{\Grind}[1]{\Gr_{#1}^{\indep} }
\nc{\GrGdom}{\curs{Gr}_G}

\nc{\GMapsRan}[1]{\mathsf{GMaps}(X,{#1})}
\nc{\GSectRan}[1]{\mathsf{GSect}({#1}/X)}
\nc{\GMapsind}[1]{\mathsf{GMaps}(X,{#1})^\indep}
\nc{\GSectind}[1]{\mathsf{GSect}({#1}/X)^\indep}
\nc{\GMapsdom}[1]{\curs{GMaps}(X,{#1})}
\nc{\GSectdom}[1]{\curs{GSect}({#1}/X)}

\nc{\chind}{\ch^{\indep}}
\nc{\chdom}{\curs{ch}}

\nc{\QSect}[1]{\curs{QSect}(#1/X)} 
\nc{\QMaps}[1]{\curs{QMaps}(X,#1)} 

\nc{\Zar}{\mathit{Zar}}

\nc{\loccit}{\textit{loc.$\,$cit.}}
\nc{\Crys}{\on{Crys}}
\nc{\ShvCatCrys}{\ShvCat^{\Crys}}

\nc{\BPE}{{\BP E}}
\nc{\BVE}{{\BV E}}
\nc{\BBE}{{\BB E}}

\nc{\Wh}{{{\CW}h}}

\nc{\ChiralCat}{\mathsf{ChiralCat}}
\nc{\RRep}{\mathfrak{R}ep}
\nc{\SSph}{\mathfrak{S}ph}

\nc{\tto}{\twoheadrightarrow}
\nc{\disj}{{\mathsf{disj}}}

\nc{\C}{\CC}
\nc{\Tch}{{\check{T}}}

\nc{\good}{\mathsf{good}}
\nc{\triv}{\mathsf{triv}}

\nc{\Alg}{\mathsf{Alg}}
\nc{\CAlg}{\mathsf{CAlg}}
\nc{\Spread}{\mathsf{Spread}}
\nc{\Dom}{\mathsf{Dom}}

\nc{\Jac}{\on{Jac}}
\renc{\CD}[1]{{#1}^{\on{CD}}}

\nc{\String}{\on{String}}

\renc{\min}{{\mathit{min}}}

\nc{\rrep}{\on-\!\mathbf{rep}}

\nc{\WWh}{\mathfrak{W}h}
\nc{\Grpd}{\mathsf{Grpd}}

\nc{\timesdisj}{\overset{\circ}\times}

\renc{\NA}{N(\sA)}
\nc{\chiral}{\mathsf{chiral}}

\nc{\Hopf}{\mathsf{Hopf}}

\nc{\heart}{\heartsuit}
\nc{\kk}{\mathbbm{k}} %% ground field

\nc{\HHom}{\CH{om}} %%Hom in DG cat.
\nc{\Cone}{\on{Cone}}

\nc{\EE}{\mathbb{E}}
\renc{\HC}{{\on{HC}}}
\nc{\HH}{{\on{HH}}}
\nc{\even}{{\on{even}}}
\nc{\SingSupp}{\on{SingSupp}}
\nc{\Supp}{\on{Supp}}

\nc{\temp}{{\on{temp}}}
\nc{\cusp}{{\on{cusp}}}
\nc{\geom}{{\on{geom}}}
\nc{\ren}{{\on{ren}}}
\nc{\naive}{{\on{naive}}}
\nc{\nnaive}{{\on{-naive}}}

\nc{\spec}{{\on{spec}}}

\nc{\gch}{\mathfrak{\check{g}}}

\nc{\Hecke}{\on{Hecke}}

\nc{\LSGch}{{\LS_\Gch}}
\nc{\LSPch}{{\LS_\Pch}}
\nc{\LSMch}{{\LS_\Mch}}

\nc{\Hsx}[2]{\H_{{#1} \leftarrow {#2}}}
\nc{\Hdx}[2]{\H_{{#1} \to {#2}}}
\nc{\Hcorr}[3]{ \H_{{#1} \leftarrow {#2} \to {#3}} }
\nc{\Hopcorr}[3]{ \H_{{#1} \to {#2} \leftto {#3}} }

\nc{\ICohsx}[2]{\ICohW_{{#1} \leftarrow {#2}}}
\nc{\ICohdx}[2]{\ICohW_{{#1} \to {#2}}}
\nc{\ICohcorr}[3]{ \ICohW_{{#1} \leftarrow {#2} \to {#3}} }
\nc{\ICohopcorr}[3]{ \ICohW_{{#1} \to {#2} \leftto {#3}} }

\nc{\QCohsx}[2]{\QCohW_{{#1} \leftarrow {#2}}}
\nc{\QCohdx}[2]{\QCohW_{{#1} \to {#2}}}
\nc{\QCohcorr}[3]{ \QCohW_{{#1} \leftarrow {#2} \to {#3}} }
\nc{\QCohopcorr}[3]{ \QCohW_{{#1} \to {#2} \leftto {#3}} }

\renc{\AA}{\bbA}
\nc{\Asx}[2]{\AA_{{#1} \leftarrow {#2}}}
\nc{\Adx}[2]{\AA_{{#1} \to {#2}}}
\nc{\Acorr}[3]{ \AA_{{#1} \leftarrow {#2} \to {#3}} }
\nc{\Aopcorr}[3]{ \AA_{{#1} \to {#2} \leftto {#3}} }

\nc{\Bsx}[2]{\B_{{#1} \leftarrow {#2}}}
\nc{\Bdx}[2]{\B_{{#1} \to {#2}}}
\nc{\Bcorr}[3]{ \B_{{#1} \leftarrow {#2} \to {#3}} }
\nc{\Bopcorr}[3]{ \B_{{#1} \to {#2} \leftto {#3}} }

\nc{\ICohzero}[3]{\ICoh_0 \bigt{#1 \times_{{#2}_\dR} #3}}
\nc{\IndCohzero}{\ICohzero}

\nc{\form}[3]{#1 \times_{{#2}_\dR} #3 }

\nc{\ind}{{\mathsf{ind}}}
\nc{\oblv}{{\mathsf{oblv}}}

\nc{\Aff}{\mathsf{Aff}}
\nc{\dgAff}{\Aff}

\nc{\deloop}{\mathsf{deloop}}
\renc{\loop}{\mathsf{loop}}

\nc{\coev}{\mathsf{coev}}

\nc{\bE}{\mathbf{E}}

\nc{\ShvCatH}{{\ShvCat^{\bbH}}}
\nc{\ShvCatQW}{\ShvCat^{\QCohW}}

\nc{\bbimod}{\on{-}\mathbf{bimod}}

\nc{\Tw}{\mathsf{Tw}}
\nc{\Arr}{\mathsf{Arr}}

\nc{\bDelta}{\bold\Delta}
\nc{\BiCat}{\mathsf{BiCat}}
\nc{\Seg}{\mathsf{Seg}}
\nc{\Cart}{\mathsf{Cart}}
\nc{\Bimod}{\mathsf{Bimod}}
\nc{\lax}{\mathit{lax}}

\nc{\pr}{\mathsf{pr}}

\nc{\zero}{ \{ 0 \}   }
\nc{\Perf}{\on{Perf}}

\nc{\leftto}{\leftarrow}
\nc{\lto}{\leftto}
\nc{\xlto}[1]{\xleftarrow{#1}}

\nc{\ltemp}{{}^\temp}
\nc{\lcusp}{{}^\cusp}
\nc{\TwCorr}{\mathsf{TwCorr}}

\nc{\Affover}[1]{{\Aff_{/#1}}}
\nc{\Affoverop}[1]{{( \Affover{#1})^\op}}

\nc{\AffOver}[2]{{(\Aff_{#1})_{/#2}}}

\nc{\AffOverop}[2]{{( \AffOver{#1}{#2})^\op}}

\nc{\aft}{{\mathit{aft}}}

\renc{\vert}{{\mathit{vert}}}
\nc{\horiz}{{\mathit{horiz}}}
\nc{\type}{{\mathit{type}}}
\nc{\adm}{{\mathit{adm}}}

\nc{\g}{\mathfrak{g}}

\nc{\free}{\mathsf{free}}

\nc{\Sform}{{S \times_{S_\dR} S}}
\nc{\Yform}{{\Y \times_{\Y_\dR} \Y}}
\nc{\SdR}{ {S_{\dR}}}

\nc{\laft}{{\mathit{laft}}}

\nc{\Affevcocaft}{\Aff_{\aft}^{< \infty}}
\nc{\Affaftevcoc}{\Aff_{\aft}^{< \infty}}

\nc{\Affevcoclfp}{\Aff_{\lfp}^{< \infty}}

\nc{\Schevcoclfp }{\Sch_{\lfp}^{< \infty}}
\nc{\Schevcocaft}{\Sch_{\aft}^{< \infty}}
\nc{\Schaftevcoc}{\Sch_{\aft}^{< \infty}}

\nc{\Stkevcoc}{\Stk^{< \infty}}
\nc{\Stkevcoclfp}{\Stk_{\lfp}^{< \infty}}
\nc{\Stkperfevcoclfp}{\Stk_{\mathit{perf},\lfp}^{< \infty}}
\nc{\Stkperflfp}{\Stk_{\mathit{perf},\lfp}}
\nc{\Stklfp}{\Stk_{\lfp}}

\nc{\evcoc}{\mathit{e.c.}}

\nc{\ICoh}{\IndCoh}

\nc{\citep}{\cite}

\renc{\H}{\bbH}

\nc{\uno}{\mathbbm{1}}

\nc{\CohBig}{{\Coh^{-\infty}}}

\nc{\Tang}{\mathbb{T}}

\nc{\LieAlg}{\mathsf{LieAlg}}

\nc{\Serre}{{\on{Serre}}}

\nc{\MPreStk}{\mathsf{MPreStk}}

\nc{\all}{{\on{all}}}

\nc{\QCohwedge}{\bbQ^\wedge}
\nc{\ICohwedge}{\bbI^\wedge}
\nc{\ICohW}{\ICohwedge}
\nc{\QCohW}{\QCohwedge}

\nc{\ShvCatA}{\ShvCat^{\AA}}
\nc{\ShvCatB}{{\ShvCat^\B}}

\nc{\naiveto}{{\xto{\naive}}}
\nc{\conaiveto}{{\xto{\conaive}}}

\nc{\strong}{\mathit{strong}}
\nc{\costrong}{\mathit{costrong}}

\nc{\conv}{\mathit{conv}}

\nc{\Q}{\bbQ}

\nc{\bY}{\mathbf{Y}}
\nc{\Loop}{\mathsf{LOOP}}

\nc{\DG}{{\on{DG}}}

\nc{\coind}{\mathsf{coind}}

\nc{\co}{{\on{co}}}

\nc{\laftdef}{{\mathit{laft-def}}}

\nc{\qsmooth}{{\mathit{qs.smooth}}}
\nc{\smooth}{{\mathit{smooth}}}

\nc{\LKE}{\on{LKE}}
\nc{\RKE}{\on{RKE}}

\nc{\ShvCatAco}{\ShvCatA_{\co}}
\nc{\ShvCatHco}{\ShvCatH_{\co}}

\nc{\Stk}{\mathsf{Stk}}

\renc{\2}{(\infty,2)}
\renc{\1}{(\infty,1)}

\nc{\doubleCat}{\mathsf{doubleCat}}
\nc{\Spaces}{\mathcal{S}\!\mathit{paces}}

\nc{\ALG}{\mathsf{ALG}}
\nc{\MAPS}{\mathsf{MAPS}}

\nc{\CAT}{\mathsf{CAT}}

\nc{\oneCat}{{\Cat_{\1}}}
\nc{\oneCAT}{{\CAT_{\1}}}

\nc{\twoCat}{{\Cat_{\2}}}
\nc{\twoCAT}{{\CAT_{\2}}}

\nc{\DGCAT}{\mathsf{DGCAT}}
\nc{\twoCatDG}{{\CAT_{\2}^\DG}}
\nc{\twoCATDG}{{\CAT_{\2}^\DG}}

\nc{\twoCATDGw}{{\CAT_{\2, w*}^\DG}}
\nc{\twoCATDGww}{{\CAT_{\2, ww*}^\DG}}

\nc{\AlgBimod}{\Alg^{\mathit{bimod}}}
\nc{\AlgBimodDGCat}{\AlgBimod(\DGCat)}
\nc{\ALGBimod}{\ALG^{\mathit{bimod}}}
\nc{\twoAlgBimod}{\ALGBimod}

\nc{\rev}{{\on{rev}}}

\nc{\lfp}{{\mathit{lfp}}}

\nc{\RBeck}{{\on{R-BC}}}
\nc{\LBeck}{{\on{L-BC}}}

\nc{\schem}{\mathit{schem}}
\nc{\proper}{\mathit{proper}}

\nc{\res}{{\mathit{res}}}

\nc{\UQCoh}{\U^{\QCoh}}
\nc{\UQ}{\UQCoh}

\nc{\LieAlgbd}{{\on{Lie-algbd}}}

\nc{\LY}{{L\Y}}

\nc{\TangQ}{\Tang^{\QCoh}}

\nc{\Fil}{{\on{Fil}}}
\nc{\AssGr}{\on{assoc-gr}}

\nc{\Cech}{\on{Cech}}

\nc{\FormMod}{\mathsf{FormMod}}
\nc{\FormModunderStkevcoc} {\FormMod_{\Stk^{< \infty}/}^\lfp }

\nc{\vDmod}{\virg{\Dmod}}
\nc{\LSG}{\LS_G}
\nc{\LSM}{\LS_M}
\nc{\LSP}{\LS_P}
\nc{\LST}{\LS_T}
\nc{\LSB}{\LS_B}
\nc{\Gm}{{\GG_m}}
\nc{\GIT}{/\!/}
\renc{\t}{\ft}

\nc{\PP}{\bbP}
\nc{\Nglob}{\check{\N}_\glob}

\nc{\Psid}{\on{Ps-Id}}
\nc{\PsId}{\Psid}
\nc{\unshift}{\Rightarrow}
\nc{\coker}{\on{cone}}

\nc{\irred}{{\on{irred}}}
\nc{\red}{{\on{red}}}
\nc{\Spr}{{\on{Spr}}}
\nc{\DL}{\mathsf{DL}}
\nc{\St}{{\on{St}}}
\nc{\Glued}{{\mathsf{Glued}}}

\nc{\LSGchLeviirred}{\LS_\Gch^{\on{Levi-irred}}}
\nc{\LSLeviirred}{\LS_G^{\on{Levi-irred}}}

\nc{\olBun}{\ol{\Bun}}
\nc{\cl}{\on{cl}}
\nc{\starext}{{*\on{-ext}}}

\nc{\RHom}{\CH \!\on{om}}

\nc{\Poinc}{\on{Poinc}}
\nc{\Betti}{\on{Betti}}
\nc{\LSGcoarse}{\LS_{G,\on{coarse}}}

\nc{\mon}{{\Gm\on{-mon}}}

\nc{\TBunG}{\sT_{\Bun_G}}
\nc{\lperp}{{}^\perp}

\nc{\Levi}{\on{Levi}}
\nc{\Verdier}{{\on{Verdier}}}

\nc{\pscpt}{{\on{ps-cpt}}}

\nc{\Calk}{\on{Calk}}

\nc{\ICohNSt}{\ICoh_{\N + \St}(\LSG)}
\nc{\NSt}{\ICohNSt}

\nc{\DmodBig}{\Dmod_{!+*}}

\nc{\Dmodproj}{\Dmod^{\on{proj}}}

\nc{\Cotrnk}{\mathsf{Cotrnk}}
\nc{\Ctrnk}{\Cotrnk}

\renc{\ss}{{\on{ss}}}
\nc{\stargen}{{*\on{-gen}}}
\nc{\shgen}{{!\on{-gen}}}

\nc{\Div}{\on{Div}}

\newcommand\blfootnote[1]{%
  \begingroup
  \renewcommand\thefootnote{}\footnote{#1}%
  \addtocounter{footnote}{-1}%
  \endgroup
}

\nc{\sR}{\mathsf{R}}

\nc{\MMaps}{\mathsf{Maps}}

\begin{document}

\title{Deligne-Lusztig duality on the stack of local systems}
\author{Dario Beraldo}

\begin{abstract}

In the setting of the geometric Langlands conjecture, we argue that the phenomenon of divergence at infinity on $\Bun_G$ (that is, the difference between $!$-extensions and $*$-extensions) is controlled, Langlands-dually, by the locus of semisimple $\Gch$-local systems.
To see this, we first rephrase the question in terms of Deligne-Lusztig duality and then study the Deligne-Lusztig functor $\DL_G^\spec$ acting on the spectral Langlands DG category $\ICoh_\N(\LSG)$.
We prove that $\DL_G^\spec$ is the projection $\ICoh_\N(\LSG) \tto \QCoh(\LSG)$, followed by the action of a coherent $D$-module $\St_G \in \Dmod(\LSG)$, which we call the \emph{Steinberg} $D$-module.
We argue that $\St_G$ might be regarded as the dualizing sheaf of the locus of semisimple $G$-local systems.
We also show that $\DL_G^\spec$, while far from being conservative, is fully faithful on the subcategory of compact objects.
\end{abstract}

\maketitle

\blfootnote{
%%%%ORCID: 0000-0002-4881-9846. \\
}

\sec{Introduction and main results}

The subjects of the present paper are:
\begin{itemize}
\item
the phenomenon of divergence at infinity on the stack $\Bun_G$;
\item
the locus of semisimple $\Gch$-local systems;
\item
the Deligne-Lusztig functors on the two sides of the geometric Langlands correspondence.
\end{itemize}
In the introduction we explain how these items are related and state our main results: Theorems \ref{mainthm:starext-ortho-to-omega}, \ref{mainthm:principal monoidal ideal}, \ref{mainthm: StG vs StM}, \ref{mainthm:St-fully faithful on CohN}, \ref{mainthm:DL factors thru QCoh}, as well as the conditional proof of Conjecture \ref{conj:stargen = semisimple}.

\ssec{Divergence at infinity on the stack of $G$-bundles}

\sssec{}

Denote by $\Bun_G := \Bun_G(X)$ the stack of $G$-bundles on a smooth complete curve $X$ defined over $\kk$. Here and always in this paper, $\kk$ denotes an algebraically closed field of characteristic zero and $G$ a connected reductive group over $\kk$.
Note that $\Bun_G$ is never quasi-compact (unless $G$ is the trivial group): by bounding the degree of instability of $G$-bundles, one obtains an exhausting sequence of quasi-compact open substacks of $\Bun_G$. 

The failure of quasi-compactness leads to the phenomenon of \emph{divergence at infinity on $\Bun_G$}, to be explained below. The goal of this paper is to describe this phenomenon from the Langlands dual point of view.

\sssec{}

We denote by $\Dmod(\Y)$ the DG category of D-modules on an algebraic stack $\Y$, see e.g. \cite{Book}. In particular, we are interested in $\Dmod(\Bun_G)$ and in its variants discussed below.

Given $U \subseteq \Bun_G$ a quasi-compact open substack, denote by $j_U$ the inclusion functor.
Let $\Dmod(\Bun_G)^{\stargen}$ be the full subcategory of $\Dmod(\Bun_G)$ generated under colimits by objects of the form $(j_U)_{*,\dR} (\F_U)$, for all quasi-compact opens $U \subseteq \Bun_G$ and all $\F_U \in \Dmod(U)$.
Similarly, let $\Dmod(\Bun_G)^{\shgen}$ be the full subcategory of 
$\Dmod(\Bun_G)$ generated under colimits by objects of the form $(j_U)_!(\F_U)$, for all quasi-compact opens $U \subseteq \Bun_G$ and all $\F_U \in \Dmod(U)$ for which $(j_U)_!(\F_U)$ is defined.

\sssec{}

It is proven in \cite{DG-cptgen} that $\Dmod(\Bun_G)^{\shgen} \simeq \Dmod(\Bun_G)$: that is, any object can be written as a colimit of $!$-extensions from quasi-compact opens.
The phenomenon of \emph{divergence at infinity on $\Bun_G$} is the fact that the inclusion $\Dmod(\Bun_G)^{\stargen} \subseteq \Dmod(\Bun_G)$ is \emph{strict}, as soon as $G$ is not abelian.
This statement is an immediate corollary of the following result, which we prove in the main body of the paper.

\begin{mainthm} \label{mainthm:starext-ortho-to-omega}
Let $G$ be a non-abelian reductive group. Any $*$-extension $(j_U)_{*,\dR} (\F_U)$ from a quasi-compact open substack $U \subset \Bun_G$ is left orthogonal to the dualizing sheaf $\omega_{\Bun_G}$: 
$$
\RHom_{\Dmod(\Bun_G)}((j_U)_{*,\dR}  (\F_U), \omega_{\Bun_G}) \simeq 0.
$$
\end{mainthm}

\begin{rem}
When $G=T$ is abelian, $\Bun_T$ is an infinite disjoint union of quasi-compact open (and closed) substacks. Thus, in this case $\Dmod(\Bun_T)^\stargen \simeq \Dmod(\Bun_T)$.
\end{rem}

\sssec{}

For $G$ non abelian (the case that we tacitly assume from now on), the above theorem allows us to exhibit many objects that do not belong to $\Dmod(\Bun_G)^{\stargen}$: as soon as $\F \in \Dmod(\Bun_G)$ is such that $\RHom(\F, \omega_{\Bun_G}) \neq 0$, there will be no way to write $\F$ as a colimit of $*$-extensions. 
Examples of such $\F$ are $\omega_{\Bun_G}$ itself, any $!$-skyscraper, or $!$-extensions of the form $(j_U)_!(\omega_U)$.

\begin{rem}
Put another way, Theorem \ref{mainthm:starext-ortho-to-omega} states that $\omega_{\Bun_G}$ is right orthogonal to $\Dmod(\Bun_G)^{\stargen}$. On the other hand, the main theorem of \cite{antitemp} states that $\omega_{\Bun_G}$ is also right orthogonal to $\Dmod(\Bun_G)^\temp$, the \emph{tempered subcategory}.\footnote{We refer to \cite{antitemp} for a detailed discussion of the full subcategory $\Dmod(\Bun_G)^\temp \subseteq \Dmod(\Bun_G)$. In the present paper, $\Dmod(\Bun_G)^\temp$ plays a role only in the introduction to motivate our research directions. To orient the reader, the main fact to know about $\Dmod(\Bun_G)^\temp$ is that, under the geometric Langlands conjecture to be recalled below, it ought to correspond to $\QCoh(\LSGch)$.}
In the case $X=\PP^1$, one can easily adapt the results of \cite{antitemp} to obtain that $\Dmod(\Bun_G(\PP^1))^{\stargen}$ and $\Dmod(\Bun_G(\PP^1))^\temp$ are equivalent as full subcategories of $\Dmod(\Bun_G(\PP^1))$.
\end{rem}

\ssec{Divergence on the Langlands dual side: the locus of semisimple $\Gch$-local systems}

\sssec{}

In view of the remark above, it might be tempting to conjecture that the $*$-generated category $\Dmod(\Bun_G)^{\stargen}$ and the tempered category $\Dmod(\Bun_G)^\temp$ are equivalent for any curve $X$. This conjecture is false in higher genus: it is shown in \cite{ramanujan} that we still have $\Dmod(\Bun_G)^{\stargen} \subseteq \Dmod(\Bun_G)^\temp$, but the inclusion is strict.
As we explain next, the difference between $\Dmod(\Bun_G)^{\stargen}$ and $\Dmod(\Bun_G)^\temp$ is accounted for, Langlands dually, by the presence of non-semisimple $\Gch$-local systems on $X$. In the case $X= \PP^1$, there is only the trivial (hence semisimple) $\Gch$-local system: this is the reason for the \virg{accidental} equivalence $\Dmod(\Bun_G(\PP^1))^{\stargen} \simeq \Dmod(\Bun_G(\PP^1))^\temp$.

\sssec{}

Recall that the geometric Langlands conjecture is supposed to match $\Dmod(\Bun_G)$ with $\ICoh_\N(\LSGch)$, where:
\begin{itemize}
\item
$\Gch$ is the Langlands dual group of $G$; 

\item
 $\LSGch$ is the derived stack of de Rham $\Gch$-local systems on $X$;
\item
$\ICoh_\N(\LSGch)$ is a certain enlargement of $\QCoh(\LSGch)$, obtained by ind-completing the full subcategory $\Coh_\N(\LSGch) \subset \QCoh(\LSGch)$ of coherent sheaves with \emph{nilpotent singular support}, see \cite{AG1}.
\end{itemize} 
We now explain how the phenomenon of divergence at infinity on $\Bun_G$ is reflected in the geometry of $\LSGch$.
For this, we need the following definition:

\begin{defn}
A $G$-local system $\sigma \in \LSG(\kk)$ is said to be \emph{semisimple} iff it is of the form $\sigma \simeq \sigma_M  \times^M G$, for some Levi subgroup $M \subseteq G$ and some irreducible $M$-local system $\sigma_M$.
Alternatively\footnote{The equivalence of these two definitions is a simple exercise using, e.g., \cite[Corollary 4.2.12]{Simon}.}, $\sigma$ is semisimple if, whenever it admits a reduction to a parabolic $P$, it admits a further reduction to the associated Levi $M$.
\end{defn}

\sssec{}

The locus $\LS^\ss_{\Gch}$ of semisimple $\Gch$-local systems is only constructible in $\LSGch$, hence the formal completion of $\LSGch$ at $\LS^\ss_{\Gch}$ does not make sense.
Nevertheless, in Section \ref{ssec:ss local systems} we will define a full subcategory $\QCoh(\LSGch)^\ss \subseteq \QCoh(\LSGch)$ which plays the role of the category of quasi-coherent sheaves on $\LSGch$ set-theoretically supported on $\LS^\ss_{\Gch}$.
With such definition, we propose:

\begin{mainconj} \label{conj:stargen = semisimple}
Under Langlands duality, $\Dmod(\Bun_G)^\stargen$ is equivalent to $\QCoh(\LSGch)^{\ss}$.
\end{mainconj}

We will \virg{prove} this conjecture in Section \ref{ssec:DL-proof of conj} by first reformulating it as Conjecture \ref{conj:spectrl T}, and then by showing that the latter follows from combining the geometric Langlands conjecture with a natural conjecture about Drinfeld's compactification of the diagonal of $\Bun_G$.

\ssec{Cuspidal objects, $\star$-extensions, and tempered objects}

This section, which can be skipped by the reader, explains how Conjecture \ref{conj:stargen = semisimple} is related to the more standard versions of the geometric Langlands conjecture.

\sssec{}

Denote by $\LS_\Gch^\irred \subset \LSGch$ the open substack of irreducible $\Gch$-local systems. Any irreducible $\Gch$-local system is obviously semisimple, hence, whatever the definition of $\QCoh(\LSGch)^\ss$, the inclusion $\QCoh(\LS_\Gch^\irred) \subseteq \QCoh(\LSGch)^\ss$ ought to hold. Indeed, once we give the definition of $\QCoh(\LSGch)^\ss$ in Section \ref{ssec:ss local systems}, the inclusion $\QCoh(\LS_\Gch^\irred) \subseteq \QCoh(\LSGch)^\ss$ will be evident.
Under Langlands duality, the chain of obvious inclusions
$$
\QCoh(\LS_\Gch^\irred) 
\subseteq
 \QCoh(\LSGch)^\ss
\subseteq
\QCoh(\LSGch)
\subseteq
\ICoh_\N(\LSGch)
$$
goes over (conjecturally) to the chain of non-obvious inclusions
$$
\Dmod(\Bun_G)^\cusp
\subseteq
\Dmod(\Bun_G)^\stargen
\subseteq
\Dmod(\Bun_G)^\temp
\subseteq
\Dmod(\Bun_G),
$$
where $\Dmod(\Bun_G)^\cusp$ is the DG category of \emph{cuspidal} D-modules on $\Bun_G$.

\sssec{}

Let us comment on the inclusions on the automorphic side.
The third inclusion is the only tautological one. The first inclusion follows from \cite[Proposition 1.4.6]{CT}. The second inclusion appears to be nontrivial:
it says that any $*$-extension is tempered. While the present paper was undergoing the publication process, a proof of this fact has become available, see \cite{ramanujan}: it relies on some of the methods of \cite{antitemp}, but it required a different point of view on the notion of temperedness.

\begin{rem}
Note that in this case an obvious fact on the spectral side, namely the inclusion $\QCoh(\LSGch)^\ss \subseteq \QCoh(\LSGch)$, informed us about something that is not evident on the automorphic side (namely, the fact that any $*$-extension is tempered).
For an instance of the inverse direction, the reader might look ahead at Theorem \ref{mainthm:St-fully faithful on CohN} and the remark following it.
\end{rem}

\ssec{Quasi-coherent sheaves on semisimple local systems} \label{ssec:ss local systems}

Let us finally give the definition of $\QCoh(\LS_G)^\ss$ and state our main results about it. (Since there is no Langlands duality in this section, we change $\Gch$ with $G$.)
We actually have two definitions: the official one, given next, and the alternative characterization provided by Theorem \ref{mainthm:principal monoidal ideal}.

\sssec{}

For a parabolic subgroup $P \subseteq G$, with Levi $M$, denote by 
$$
\LSM \xto{\i_{P}} \LSP \xto{\p_P} \LSG
$$
the induction functors.
In spite of the notation, the map $\i_P: \LSM \to \LSP$ is not at all an embedding. Yet, by the contraction principle, the functor $(\i_P)_{*,\dR}: \Dmod(\LSM) \hto \Dmod(\LSP)$ is fully faithful.
We will recall the contraction principle in Section \ref{ssec:Braden}.

\sssec{}

As a preliminary step, we define the full subcategory $\Dmod(\LSG)^\ss \subseteq \Dmod(\LSG)$ of D-modules \emph{supported on $\LSG^\ss$}: an object $\F \in \Dmod(\LSG)$ belongs to $\Dmod(\LSG)^\ss$ iff
$$
(\p_P)^{!,\dR} (\F) 
\in
(\i_P)_{*,\dR} \bigt{ \Dmod(\LSM) }
,
\hspace{.2cm} \mbox{for any $P$}.
$$
Note that such definition mimics the definition of semisimple $G$-local systems: $\sigma$ is semisimple iff, whenever it is reducible to $P$, it is also reducible to $M$.
Next, define $\QCoh(\LSG)^\ss$ to be the cocompletion of the essential image of the action functor
$$
\QCoh(\LSG)
\usotimes{\Dmod(\LSG)}
\Dmod(\LSG)^\ss
\longto
\QCoh(\LSG).
$$

\sssec{}

Here are some facts that support this definition:

\begin{itemize}
\item
If $\F \in \Dmod(\LSG)^\ss$, then it is immediately checked that its $(!, \dR)$-fiber at $\sigma \in \LSG(\kk)$ is zero whenever $\sigma$ is not semisimple.
Similarly, if $\F \in \QCoh(\LSG)^\ss$, then its $*$-fiber at $\sigma \in \LSG(\kk)$ is zero whenever $\sigma$ is not semisimple.

\item
In Section \ref{sssec:defn of St}, we will define the \emph{Steinberg}
D-module $\St_G$ and its underlying quasi-coherent sheaf $\ul\St_G \in \QCoh(\LSG)$. The former plays the role of the dualizing sheaf on the non-existent stack $\LSG^\ss$, the latter plays the role of the structure sheaf of the non-existent formal completion $(\LSG)^\wedge_{\LSG^\ss}$. For instance, the geometric fibers $\restr{\ul\St_G}{\sigma}$ are zero for $\sigma$ non-semisimple and $1$-dimensional (but sitting in varying cohomological degree) if $\sigma$ is semisimple.

\item
As a consequence of Theorem \ref{mainthmbis: fibers of St}, any skyscraper (in either the D-module or the quasi-coherent sense) at a semisimple local system belongs to $\Dmod(\LSG)^\ss$ or $\QCoh(\LSG)^\ss$.
\end{itemize}

\sssec{}

The next two theorems provide alternative characterizations of $\QCoh(\LSG)^\ss$ and $\Dmod(\LSG)^\ss$.
Given a cocomplete monoidal symmetric DG category $(\C,\otimes)$, recall the notion of \emph{principal monoidal ideal generated by $c\in\C$}: this is the full subcategory of $\C$ consisting of objects in the essential image of $c \otimes - : \C \to \C$. Note that a principal monoidal ideal might not be closed under colimits.

\begin{mainthm} \label{mainthm:principal monoidal ideal}
The full subcategory $\QCoh(\LS_G)^\ss \subseteq \QCoh(\LS_G)$ is a \emph{principal monoidal ideal}, generated by the Steinberg object $\ul\St_G$ (defined below).
\end{mainthm}

\begin{mainthmbis}{mainthm:principal monoidal ideal} 
[Divisibility by the Steinberg D-module] \label{mainthmbis:Dmod ss principal ideal}
The full subcategory $\Dmod(\LSG)^\ss \subseteq \Dmod(\LSG)$ is a principal monoidal ideal generated by $\St_G \in \Dmod(\LSG)$.
More precisely, the functor $\St_G \sotimes -: \Dmod(\LSG) \to \Dmod(\LSG)$ lands in $\Dmod(\LSG)^\ss$ and it comes with an explicit right inverse $\Div_G: \Dmod(\LSG)^\ss \to \Dmod(\LSG)$.
\end{mainthmbis}

\begin{rem}
In Section \ref{sec:divisibilty by St}, we will spell out the definition of $\Div_G$, prove Theorem \ref{mainthmbis:Dmod ss principal ideal} and then deduce Theorem
\ref{mainthm:principal monoidal ideal} from it. 
The statement of Theorem \ref{mainthmbis:Dmod ss principal ideal}, which we called \virg{divisibility by the Steinberg object} was inspired by \cite[Theorem 1.1]{Lu}. The latter theorem shows that, in the modular representation theory of finite groups of Lie type, the ideal of projective representations is principal and generated by the Steinberg module. 
\end{rem}

\begin{rem}
Since $\QCoh(\LS_G)^\ss$ is by its very definition closed under colimits, Theorem \ref{mainthm:principal monoidal ideal} implies that the same holds true for the principal monoidal ideal generated by $\ul\St_G$. This is a consequence of Theorem \ref{mainthm:St-fully faithful on CohN}, see Lemma \ref{lem:ulSt a single tensor} for the proof.

On the other hand, in the D-module setting of Theorem \ref{mainthmbis:Dmod ss principal ideal}, the existence of $\Div_G$ immediately implies that the principal monoidal ideal generated by $\St_G$ is closed under colimits.
\end{rem}

\sssec{} \label{sssec:defn of St}

Let us now define the Steinberg objects $\St_G \in \Dmod(\LSG)^\ss$ and $\ul\St_G \in \QCoh(\LSG)^\ss$. As mentioned, they are precise analogues of a well-known object of classical representation theory (see, e.g., \cite{St}, \cite{curtis}, \cite{Lu}, \cite{hump}). By definition, $\St_G$ is the coherent D-module defined as
\begin{equation} \label{eqn:defn of St D-module}
\St_G:= \coker
\Bigt{
\uscolim{P \in \Par'} \;
(\p_P)_{*,\dR}(\omega_{\LSP,\dR})
\to
\omega_{\LSG,\dR}
}
\in \Dmod(\LSG),
\end{equation}
where $\Par'$ be the poset of proper standard (relative to a chosen Borel $B$, fixed throughout) parabolics of $G$.
Then we set $\ul\St_G$ to be the quasi-coherent sheaf underlying the Steinberg D-module $\St_G \in \Dmod(\LSG)$.

\sssec{}

From the above formula, it is not even clear that $\St_G$ belongs to $\Dmod(\LSG)^\ss$, let alone a generator.
To prove that $\St_G \in \Dmod(\LSG)^\ss$ (and consequently that $\ul\St_G$ belongs to $\QCoh(\LSG)^\ss$), we will show that the Steinberg construction enjoys the following \virg{functional equation}, which relates $\St_G$ to the Steinberg object for a Levi subgroup $M \subseteq G$.

\begin{mainthm} \label{mainthm: StG vs StM}
For any parabolic $P \subseteq G$ with Levi $M$, there is a canonical isomorphism
\begin{equation} \label{intro-eqn: StG vs StM}
\p_P^{!,\dR}(\St_G) 
\simeq 
(\i_P)_{*,\dR}(\St_M)[\rk(G)-\rk(M)]
\end{equation}
in $\Dmod(\LSP)$.
\end{mainthm}

\sssec{}

Formula \eqref{intro-eqn: StG vs StM} allows to compute the $!$-fibers of $\St_G$:

\begin{mainthmbis}{mainthm: StG vs StM} \label{mainthmbis: fibers of St}
For $\sigma \in \LSG$ a $\kk$-point, $\restr {\St_G} \sigma \neq 0$ if and only if $\sigma$ is semisimple. If $\sigma \simeq \sigma_M \times^M G$ with $\sigma_M$ irreducible, then 
\begin{equation} \label{eqn:fibers of StG at sigmaM}
\restr {\St_G} \sigma
 \simeq
 \kk[2 \dim H^0(X_\dR, \fu_{\sigma_M}) + \rk(G) - \rk(M)],
\end{equation}
where $\fu$ is the Lie algebra of the unipotent radical of a parabolic $P$ with Levi $M$. 
\end{mainthmbis}

\begin{example}
In particular, $\restr{\St_G}{\sigma} \simeq \kk$ when $\sigma$ is irreducible. This is obvious from the definition, as $\restr{\St_G}{\LSG^\irred} \simeq \omega_{\LSG^\irred}$.
\end{example}

\begin{example}
At the other extreme, for $\sigma_\triv$ the trivial $G$-local system, \eqref{eqn:fibers of StG at sigmaM} yields
$$
\restr \St {\sigma_\triv} \simeq \kk[|\sR|+\rk(G)],
$$
where $|\sR|$ is the number of roots of $G$. This can also be seen directly: indeed, 
$$
\restr {\St_G} {\sigma_\triv} 
\simeq
\coker
\Bigt{ 
\uscolim{P \in \Par'} \, H_*(G/P) 
 \to \kk},
$$
and the latter can be simplified with the help of some standard Weyl combinatorics.
\end{example}

\sssec{}

The functor $\ul\St_G \otimes - : \QCoh(\LSG) \to \QCoh(\LSG)^\ss$ is very far from being an inclusion: as shown above, it annihilates skyscrapers at non-semisimple $G$-local systems, so it is not even conservative. Thus, the next result comes perhaps as a surprise.

Recall that $\Coh_\N(\LSG)$ denotes the (non-cocomplete) full subcategory of $\QCoh(\LSG)$ consisting of coherent sheaves with nilpotent singular support (see \cite{AG1}).

\begin{mainthm} \label{mainthm:St-fully faithful on CohN}
The functor 
$$
\ul\St_G \otimes - : 
\Coh_\N(\LSG)
\longto
\QCoh(\LSG)^\ss
$$
is fully faithful.
\end{mainthm} 

Thus, $\QCoh(\LSG)^\ss$ contains a \virg{hidden} copy of $\Coh_\N(\LSG)$, which we will later denote as $\Coh_\N^\St(\LSG)$.

\begin{rem}
This statement is the Langlands dual of an evident statement on the automorphic side: the fact that the composition of the miraculous and the naive duality is fully faithful when restricted to compact objects. 
We will explain this, as well as the relation between miraculous duality, Deligne-Lusztig duality and the Steinberg object, in the next section.
\end{rem}

\ssec{Deligne-Lusztig duality and the proof of Conjecture \ref{conj:stargen = semisimple}} \label{ssec:DL-proof of conj}

\sssec{}

Let $\C$ be a dualizable cocomplete DG category. Recall that functors from $\C^\vee \to \C$ are given by \virg{kernels} in $\C \otimes \C$.
In the case $\C = \Dmod(\Y)$ with $\Y$ a quasi-compact\footnote{The correct technical condition is QCA, see \cite{finiteness}.} algebraic stack, the kernel 
$$
\Delta_{*,\dR}(\omega_\Y) 
\in 
\Dmod(\Y \times \Y)
\simeq 
\Dmod(\Y) \otimes \Dmod(\Y)
$$
provides a self-duality equivalence 
$$
\Psid_*: \Dmod(\Y)^\vee \longto \Dmod(\Y).
$$
When $\Y$ is not quasi-compact, $\Psid_*$ is not an equivalence, unless the closure of any quasi-compact open of $\Y$ is itself quasi-compact (see \cite{DG-cptgen}). 
In particular, $\Psid_*: \Dmod(\Bun_G)^\vee \to \Dmod(\Bun_G)$ is never an equivalence when $G$ is not abelian.

\sssec{}

On the other hand, following V. Drinfeld, one defines 
$$
\Psid_!: \Dmod(\Y)^\vee \longto \Dmod(\Y)
$$
to be the functor determined by the kernel 
$$
\Delta_!(k_\Y) 
\in 
\Dmod(\Y \times \Y)
\simeq 
\Dmod(\Y) \otimes \Dmod(\Y).
$$
The stack $\Y$ is said to be \emph{miraculous} if $\Psid_!$ is an equivalence. By \cite{DG-cptgen}, $\Bun_G$ is miraculous; moreover it contains an exhausing sequence of miraculous quasi-compact opens.

\sssec{}

Let us consider the composition of the miraculous and the naive duality, that is, the functor
$$
\TBunG := \Psid_* \circ \Psid_!^{-1}:
\Dmod(\Bun_G)
\longto
\Dmod(\Bun_G).
$$
The essential image of $\TBunG$ is easy to identify and relevant to our discussion: indeed, in Section \ref{ssec:miraculous}, we will prove that
\begin{equation} \label{eqn:image of TBunG}
\TBunG(\Dmod(\Bun_G)) \simeq \Dmod(\Bun_G)^\stargen.
\end{equation}

\sssec{}

Thus, Conjecture \ref{conj:stargen = semisimple} is an immediate corollary of the following one.
Let $d_G := 2 \dim(\Bun_G) + \dim(Z_G)$.

\begin{mainconjbis}{conj:stargen = semisimple} \label{conj:spectrl T}
Under Langlands duality, the functor $\TBunG[- d_G]$ goes over to the functor
$$
\ICoh_\N(\LSGch)
\stackrel \Psi \tto
\QCoh(\LSGch)
\xto{\ul\St_\Gch \otimes - }
\QCoh(\LSGch)^\ss
\hto
\QCoh(\LSGch)
\stackrel \Xi \hto
\ICoh_\N(\LSGch),
$$
where
$$ %%%best-colocalization-diagram-prototype
\begin{tikzpicture}[scale=1.5]
\node (a) at (0,1) {$\QCoh(\LS_G) $};
\node (b) at (2.5,1) {$\ICoh_\N(\LS_G)$};
\path[right hook ->,font=\scriptsize,>=angle 90]
([yshift= 1.5pt]a.east) edge node[above] {$\Xi$} ([yshift= 1.5pt]b.west);
\path[->>,font=\scriptsize,>=angle 90]
([yshift= -1.5pt]b.west) edge node[below] {$\Psi$} ([yshift= -1.5pt]a.east);
\end{tikzpicture}
$$
is the standard adjunction. In short: the functor $\TBunG[-d_G]$ is Langlands dual to the composition of temperization with the action by $\St_\Gch$.
\end{mainconjbis}

\sssec{}

Let us explain how this statement ought to follow from the Langlands conjecture.
On the automorphic side, it was conjectured that
\begin{equation}  \label{eqn:T = DL}
\TBunG[-d_G]
\simeq
\DL_G, 
\end{equation}
where $\DL_G$ is the \emph{Deligne-Lusztig} functor 
$$
\DL_G
:=
\coker
 \Bigt{  
 \uscolim{P \in \Par'} \,  \Eis_P^\enh \circ \CT_P^\enh 
 \longto
 \id_{\Dmod(\Bun_G)}
  }
 :
 \Dmod(\Bun_G)
 \longto
 \Dmod(\Bun_G).
$$
Here, the functors $\Eis_P^\enh$ and $\CT_P^\enh$ are the \emph{enhanced} Eisenstein series and constant term functors, see \cite[Section 6.3]{Outline}. We will not need their definition, hence we do not recall it.

\begin{rem}
Let us briefly review the history of \eqref{eqn:T = DL}. For $G=SL_2$, this was conjectured by V. Drinfeld and J. Wang, see \cite[Appendix C]{DW}. For $G$ arbitrary, we learned the statement from a private communication with D. Gaitsgory; the statement essentially appears in \cite[Section 6.6 and especially Remark 6.6.5]{Wang}. Finally, while the present paper was undergoing the publication process, \eqref{eqn:T = DL} was proven by L. Chen in \cite{LinChen}.
\end{rem}

In view of \cite{LinChen}, the following remark is obsolete; we leave it for reference.

\begin{rem} 
As we learned from D. Gaitsgory, one way to prove \eqref{eqn:T = DL} goes by expressing the LHS via $\ol\Bun_G$, Drinfeld's compactification of the diagonal of $\Bun_G$. It is known that $\ol\Bun_G$ is naturally stratified by $\Par$, the (open) $G$-stratum yielding the identity functor. The question is then to prove that the $P$-stratum yields the functor $ \Eis_P^\enh \circ \CT_P^\enh$.
\end{rem}

\sssec{}

The postulated compatibility of geometric Langlands with enhanced constant terms and Eisenstein series, see \cite[Sections 6.6.4 and 6.6.5]{Outline}, predicts that $\DL_G$ corresponds to the similar looking functor on the spectral side:
$$
\DL_\Gch^\spec
:=
\coker
 \Bigt{  \uscolim{P \in \Par'} \, \Eis_\Pch^{\enh, \spec} \circ \CT_{\Pch}^{\enh, \spec} \longto \id_{\ICoh_\N(\LSGch)} }
 :
\ICoh_\N(\LSGch)
 \longto
\ICoh_\N(\LSGch).
$$
In this case, we do need the definitions of $\Eis_\Pch^{\enh, \spec}$ and $\CT_{\Pch}^{\enh, \spec}$: they are recalled in Section \ref{ssec:prelims}.
Using the techniques of \cite{AG2} and \cite{strong-gluing}, we will be able to simplify the functor $\DL_\Gch^\spec$ to obtain:

\begin{mainthm} \label{mainthm:DL factors thru QCoh}
The functor $\DL_G^\spec$ decomposes as
$$
\ICoh_\N(\LSG)
\stackrel \Psi \tto
\QCoh(\LSG)
\xto{\ul{\St}_G \otimes - }
\QCoh(\LSG)
\stackrel{\Xi}\hto
\ICoh_\N(\LSG).
$$
\end{mainthm}

\sssec{}

With this theorem proven, the assertion of Conjecture \ref{conj:spectrl T} is a corollary (modulo the geometric Langlands conjecture) of the combination of \eqref{eqn:image of TBunG} and Theorem \ref{mainthm:principal monoidal ideal}.

\begin{rem}

In the course of the proof of \eqref{eqn:image of TBunG}, we will see that, while $\TBunG$ is not even conservative, it is nevertheless fully faithful on compact objects. Hence, the same property must be true for $\DL_G$ and $\DL_G^\spec$. Combining this with the statement of Theorem \ref{mainthm:DL factors thru QCoh} led us to the statement of Theorem \ref{mainthm:St-fully faithful on CohN}.
\end{rem}

\ssec{Restoring the \virg{duality}}

Theorem \ref{mainthm:DL factors thru QCoh} implies that the Deligne-Lusztig functor $\DL_G^\spec$ is not a duality. However, Theorem \ref{mainthm:St-fully faithful on CohN} suggests a way to modify $\DL_G^\spec$ to make it into an equivalence.

\sssec{}

Let
$$
\Coh_\N^\St(\LSG) := \ul\St_G \otimes \Coh_\N(\LSG)
\subseteq \QCoh(\LSG).
$$
In other words, $\Coh_\N^\St(\LSG)$ is the essential image of the fully faithful functor appearing in Theorem \ref{mainthm:St-fully faithful on CohN}.
We also define
$$
\ICoh^\St_\N(\LSG)
:= \Ind(\Coh_\N^\St(\LSG)).
$$
This DG category comes with a tautological essentially surjective functor $\Psi_\St: \ICoh^\St_\N(\LSG) \to \QCoh(\LSG)^\ss$, induced by inclusion $\Coh_\N^\St(\LSG) \subseteq \QCoh(\LSG)^\ss$.
 
 \sssec{}
 
 Theorem \ref{mainthm:St-fully faithful on CohN} shows that the action of $\St_G$ yields an equivalence
$\DL_G^{\spec, \enh}: \ICoh_\N(\LSG) \simeq \ICoh_\N^\St(\LSG)$, which is ought to be Langlands dual to the  inverse of the miraculous duality. Likewise, $\Psi_\St$ is Langlands dual to the naive duality.

\medskip

Theorem \ref{mainthm:DL factors thru QCoh} shows that the square
$$ 
\begin{tikzpicture}[scale=1.5]
\node (01) at (0,1.2) {$\ICoh_\N(\LSG)$};
\node (11) at (3,1.2) {$\ICoh_\N^\St(\LSG)$};
\node (10) at (3,0) {$\QCoh(\LSG)^\ss$};
\node (00) at (0,0) {$\QCoh(\LSG)$};
\path[->,font=\scriptsize,>=angle 90]
(00.east) edge node[above] {$\ul\St_G \otimes -$} (10.west);
\path[->>,font=\scriptsize,>=angle 90]
(01.south) edge node[left] {$\Psi$}  (00.north);
\path[->,font=\scriptsize,>=angle 90]
(01.east) edge node[above] {$\DL_G^{\spec, \enh}$} node[below] {$\simeq$} (11.west);
\path[->,font=\scriptsize,>=angle 90]
(11.south) edge node[right] {$\Psi_\St$} node[left] {} (10.north);
\path[->,font=\scriptsize,>=angle 90]
(01.south east) edge node[right] {$\; \DL_G^\spec$} node[left] {} (10.north west);
\end{tikzpicture}
$$
is commutative. Langlands dually (and changing $G$ with $\Gch$), the above commutative diagram ought to read as
$$ 
\begin{tikzpicture}[scale=1.5]
\node (01) at (0,1.2) {$\Dmod(\Bun_G)$};
\node (11) at (3,1.2) {$\Dmod(\Bun_G)^\vee$};
\node (10) at (3,0) {$\Dmod(\Bun_G)^\stargen$,};
\node (00) at (0,0) {$\Dmod(\Bun_G)^\temp$};
\path[->,font=\scriptsize,>=angle 90]
(00.east) edge node[above] {$\ul\St_G \otimes -$} (10.west);
\path[->>,font=\scriptsize,>=angle 90]
(01.south) edge node[left] {$\temp$}  (00.north);
\path[->,font=\scriptsize,>=angle 90]
(01.east) edge node[above] {$\Psid_!^{-1}$} node[below] {$\simeq$} (11.west);
\path[->,font=\scriptsize,>=angle 90]
(11.south) edge node[right] {$\Psid_*[-d_G]$} node[left] {} (10.north);
\path[->,font=\scriptsize,>=angle 90]
(01.south east) edge node[right] {$\; \DL_G$} node[left] {} (10.north west);
\end{tikzpicture}
$$
where the tensor product on the bottom line denotes the action of $\QCoh(\LSGch)$ on $\Dmod(\Bun_G)$ given by the vanishing theorem of \cite[Section 4.5]{Outline}.

\ssec{Compatibility with Eisenstein series}

To conclude the introduction, we ask how the enhanced Deligne-Lusztig duality interacts with Eisenstein series. (This can be skipped by the reader, as it will not be used anywhere in the paper.) 
In other words, we wish to describe the rightmost vertical functor in the following commutative diagram:
$$ 
\begin{tikzpicture}[scale=1.5]
\node (01) at (0,1) {$\ICoh_\N(\LSG)$};
\node (11) at (3,1) {$\ICoh_\N^\St(\LSG)$};
\node (00) at (0,0) {$\ICoh_\N(\LSM)$};
\node (10) at (3,0) {$\ICoh_\N^\St(\LSM)$.};
\path[->,font=\scriptsize,>=angle 90]
(00.east) edge node[above]  {$\DL_M^\enh$}  node[below] {$\simeq$} (10.west);
\path[<-,font=\scriptsize,>=angle 90]
(01.south) edge node[left] {$\Eis_P$}  (00.north);
\path[->,font=\scriptsize,>=angle 90]
(01.east) edge node[above] {$\DL_G^{\enh}$} node[below] {$\simeq$} (11.west);
\path[<-,font=\scriptsize,>=angle 90]
(11.south) edge node[right] {$\Eis_P^\St$} node[left] {} (10.north);
\end{tikzpicture}
$$
To this end, consider the functor\footnote{As usual, the notation $\ul \F$ denotes the quasi-coherent sheaf underlying the D-module $\F$.}
$$
\QCoh(\LSM) \longto \QCoh(\LSG),
\hspace{.5cm}
\F_M
\squigto
(\p_P)_* 
\Bigt{
\ul{
(\i_P)_{ *,\dR} (\omega_{\LSM,\dR})
}
\otimes
\q_P^*(\F_M)
}.
$$
Theorem \ref{mainthm: StG vs StM} shows that such functor sends $\Coh_\N^\St(\LSM)$ to $\Coh_\N^\St(\LSG)$. Ind-completing, we obtain a functor $\Eis_P^\St$ that makes the square commutative by inspection.

\ssec{Structure of the paper}

The rest of the paper is devoted to proving our main results, in a different order than the one presented in the introduction: Theorem \ref{mainthm:starext-ortho-to-omega} in Section \ref{sec:divergence}, Theorem \ref{mainthm:principal monoidal ideal} in Section \ref{sec:divisibilty by St}, Theorem \ref{mainthm: StG vs StM} in Section \ref{sec:funct eqn St}, Theorem \ref{mainthm:St-fully faithful on CohN} in Section \ref{sec:St fully faith on CohN} and Theorem \ref{mainthm:DL factors thru QCoh} in Section \ref{sec: DL spectral sends all to temp}.

\ssec{Conventions and notation} \label{ssec: notation}

We will mostly use the conventions of \cite[Section 2]{antitemp} and \cite{strong-gluing}.

\sssec{}

To shorten formulas, we use the notation $\ul{\M} := \oblv_L(\M)$ to indicate the quasi-coherent sheaf underlying a $\Dmod$-module $\M$. Here, $\oblv_L: \Dmod(-) \to \QCoh(-)$ denotes the \virg{left} forgetful functor, from D-modules to quasi-coherent sheaves. There will be also a \virg{right} forgetful functor $\oblv_R: \Dmod(-) \to \ICoh(-)$.
We regard $\Dmod(\Y)$ as a symmetric monoidal DG category under $\sotimes$.
When dealing with D-modules, we often write $f_*$ instead of the more precise $f_{*,\dR}$, hoping that the real meaning will be clear from the context. For instance, in the expressions $\RHom_{\Dmod(\Y)}(f_*(\M), \N)$ and $\ul{f_*(\M)}$, it should be clear that both push-forwards are de Rham ones.

\ssec{Acknowledgements}

I am grateful to I. Grojnowski for several useful discussions and K. McGerty for referring me to the paper \cite{Lu}, which prompted Theorem \ref{mainthm:principal monoidal ideal}. 
Thanks also to D. Gaitsgory and B. To\"en for help with the notion of semisimplicity for local systems.
Research supported by ERC-2016-ADG-741501.

\sec{Divergence at infinity on $\Bun_G$} \label{sec:divergence}

In this section, we give details on the phenomenon of divergence at infinity on the stack $\Bun_G$ and prove Theorem \ref{mainthm:starext-ortho-to-omega}.

\ssec{Miraculous duality, $!$-extensions, $*$-extensions} \label{ssec:miraculous}

\sssec{}

It is established in \cite{DG-cptgen} that any quasi-compact open substack of $\Bun_G$ is contained in a quasi-compact open substack $U$ with the following remarkable property: the $!$-pushforward $(j_U)_!$ along the open embedding $j_U: U \hto \Bun_G$ is well-defined on the entire $\Dmod(U)$.
Quasi-compact opens of $\Bun_G$ with this property are called \emph{cotruncative}.
The actual construction of such open substacks is not important for us: we refer to \cite{DG-cptgen} for details.

We denote by $\Cotrnk$ the $1$-category of cotruncative open substacks of $\Bun_G$; any finite union of cotruncative substacks is cotruncative, so that $\Cotrnk$ is filtered.

\sssec{}

Another property of $\Bun_G$ of similar kind is the fact that the functor $(p_{\Bun_G})_! : \Dmod(\Bun_G) \to \Vect$ is well-defined. This follows from the contractibility of the space of rational maps into $G$, together with the ind-properness of the Beilinson-Drinfeld Grassmannian (see \cite[Corollary 5.3.2]{contract} for details).

\sssec{Terminology}

When we say that $\F \in \Dmod(\Bun_G)$ is a $!$-extension, we mean that there exist a \emph{quasi-compact open} $U$ such that $\F \simeq (j_U)_!(j_U^! \F)$. Without loss of generality, we can assume such $U$ to be cotruncative.
The term $*$-extension is used accordingly.

\sssec{}

It is clear that $\Dmod(\Bun_G)$ is generated by $!$-extensions, that is,
$$
\Dmod(\Bun_G)
\simeq
\uscolim{U \in \Cotrnk} 
\;
(j_U)_! (\Dmod(U)).
$$
Moreover, any compact object of $\Dmod(\Bun_G)$ is of the form $(j_U)_!(\F_U)$ for some $U \in \Cotrnk$ and some compact $\F_U$.

\sssec{}

As already discussed in the introduction, denote by $\Dmod(\Bun_G)^\stargen$ the full subcategory of $\Dmod(\Bun_G)$ generated under colimits by $*$-extensions. Note that the functor $(j_U)_{*} : \Dmod(U) \to \Dmod(\Bun_G)$ does not preserves compactness in general. 

\sssec{}

Recall now the miraculous duality of $\Bun_G$ and the functor $\TBunG := \Psid_* \circ \Psid_!^{-1}$. It is proven in \cite[Lemma 4.5.7]{DG-cptgen} that any cotruncative open substack of $\Bun_G$ is also miraculous. For any QCA stack $\Y$, the functor $\Psid_{\Y,*}$ is an equivalence: this is our standard way to identify $\Dmod(\Y)$ with its dual. Hence, for $U \in \Cotrnk$ (and in fact for any miraculous QCA stack), we regard the functor $\Psid_{U,!}$ as a self equivalence of $\Dmod(U)$.

\sssec{} \label{sssec:TBunG is fully faithful on pairs (any, !-ext)}

Thanks to \cite[Lemma 4.4.12]{DG-cptgen}, for any $U \in \Cotrnk$, we have
\begin{equation}\label{eqn:TBunG on !-ext}
\TBunG ((j_U)_!(\F_U))
\simeq
(j_U)_{*}  (\Psid_{U,!}^{-1}(\F_U)).
\end{equation}
It follows that $\TBunG$ is fully faithful on $!$-extensions (in particular: on compact objects), and thus, by taking colimits in the first variable, fully faithful on pairs $(\on{any}, \on{!-ext})$. The latter means that, for any $\F \in \Dmod(\Bun_G)$ and any $*$-extension $ (j_U)_!(\F_U)$, the functor $\TBunG$ yields an isomorphism
$$
\RHom_{\Dmod(\Bun_G)} (\F,(j_U)_!(\F_U) )
\xto{\;\; \simeq \;\;}
\RHom_{\Dmod(\Bun_G)} \bigt{ 
\TBunG(\F), \TBunG ((j_U)_!(\F_U)) 
}.
$$

\begin{rem}
On the other hand, $\TBunG$ is not fully faithful on the entire $\Dmod(\Bun_G)$. In fact, it is not even conservative, as
$$
\TBunG(\omega_{\Bun_G}) \simeq 0.
$$
To show this, follow the argument of \cite{ker-adj} and invoke \cite[Corollary 1.4.2]{antitemp} when proving that $\omega_{\Gr_G}$ is infinitely connective.
\end{rem}

\begin{lem}
The essential image of $\TBunG$ equals $\Dmod(\Bun_G)^\stargen$.
\end{lem}

\begin{proof}
Any object of $\Dmod(\Bun_G)$ is a colimit of $!$-extensions from cotruncative (hence miraculous) open substacks: \eqref{eqn:TBunG on !-ext} then shows that the essential image is contained in $\Dmod(\Bun_G)^\stargen$.
By the same formula, any $*$-extension belongs to the essential image of $\TBunG$. It remains to show that such essential image is closed under colimits. In other words, we need to show that, for any index $\infty$-category $\I$ and any functor 
$$
\I \to \Dmod(\Bun_G),
\hspace{.4 cm}
i \squigto (j_{U_i})_*(\F_i),
$$
there exists an object $\F$ such that $\TBunG(\F) \simeq \colim_i  (j_{U_i})_*(\F_i)$.
Without loss of generality, we can assume that each $U_i$ is cotruncative. Then the assertion follows from the fully faithfulness of $\TBunG$ on $!$-extensions.
\end{proof}

\ssec{Proof of Theorem \ref{mainthm:starext-ortho-to-omega}}

The following result shows that the inclusion $\Dmod(\Bun_G)^\stargen \subseteq \Dmod(\Bun_G)$ is actually very strict (for $G$ non-abelian): any object of $\Dmod(\Bun_G)^\stargen$ has no de Rham cohomology with compact supports.

\begin{thm}
Let $G$ be a reductive group of semisimple rank $\geq 1$. For any quasi-compact open $U \subset \Bun_G$, the functor $(p_{\Bun_G})_! \circ (j_U)_{*} : \Dmod(U) \to \Vect$ is identically zero.
\end{thm}

Since $(p_{\Bun_G})_!$ is left adjoint to $(p_{\Bun_G})^!$, this theorem is equivalent to Theorem \ref{mainthm:starext-ortho-to-omega}.

\begin{proof}
We proceed in six steps.

\sssec*{Step 1}

Without loss of generality, we may assume that $U$ is cotruncative.
By adjunction, we need to show that $\RHom_{\Dmod(\Bun_G)}((j_{U})_*\F_U, \omega_{\Bun_G}) \simeq 0$ for any $\F_U \in \Dmod(U)^\cpt$.
Tautologically, we have
$$
\RHom_{\Dmod(\Bun_G)}((j_{U})_*\F_{U}, \omega_{\Bun_G})
\simeq
\lim_{U' \in (\Cotrnk_{U/})^\op}
\RHom_{\Dmod(U')}
\bigt{
(j_{U \to U'})_*\F_U,
\omega_{U'}
},
$$
where $j_{U \to U'}: U \hto U'$ is the structure inclusion.

\sssec*{Step 2}

Now, note that the functor $(j_{U \to U'})_*: \Dmod(U) \to \Dmod(U')$ admits a continuous right adjoint, which will be denoted by $(j_{U \to U'})^?$. This follows from the definition of cotruncativeness: indeed, the functor $(j_{U \to U'})_!$ is clearly defined and $(j_{U_0 \to U})^?$ is tautologically its dual (under the standard self dualities of the DG category of D-modules on a QCA stack, see \cite{finiteness}).

\sssec*{Step 3}
Hence, 
\begin{eqnarray}
\nonumber
\RHom_{\Dmod(\Bun_G)}((j_{U})_*\F_{U}, \omega_{\Bun_G})
& \simeq &
\lim_{U' \in (\Cotrnk_{U/})^\op }
\RHom_{\Dmod(U)}
\bigt{
\F_U,
(j_{U \to U'})^?
\omega_{U'}
}
\\
\nonumber
& \simeq &
\RHom_{\Dmod(U)}
\Bigt{
\F_U,
\lim_{U' \in (\Cotrnk_{U/})^\op }
(j_{U \to U'})^?
\omega_{U'}
}.
\end{eqnarray}
Thus, the theorem is equivalent to proving that, for any $U$, we have:
$$
\lim_{U' \in (\Cotrnk_{U/})^\op }
(j_{U \to U'})^?
\omega_{U'}
\simeq 0.
$$

\sssec*{Step 4}

Let $k_{\Bun_G}$ be the constant sheaf on $\Bun_G$, that is, the Verdier dual of $\omega_{\Bun_G}$. By smoothness, we have $k_{\Bun_G}[2 \dim(\Bun_G)] \simeq \omega_{\Bun_G}$.
We claim that 
\begin{equation} \label{eqn:Hom from k}
\RHom_{\Dmod(\Bun_G)}
\bigt{
k_{\Bun_G}, (j_{U})_! \F_U
}
\simeq
0.
\end{equation}
This is immediate from the discussion of Section \ref{sssec:TBunG is fully faithful on pairs (any, !-ext)} and the remark following it.

\sssec*{Step 5}

Starting from  \eqref{eqn:Hom from k}, we obtain that 
$$
0
\simeq
\RHom_{\Dmod(\Bun_G)}
\bigt{
k_{\Bun_G}, (j_{U})_! \F_U
}
\simeq
\lim_{U' \in (\Cotrnk_{U/})^\op }
\RHom_{\Dmod(U)}
\bigt{
k_{U'},
(j_{U \to U'})_!
\F_{U}
}.
$$
The objects appearing on the RHS are all coherent: hence, we can apply Verdier duality to obtain
$$
\lim_{U' \in (\Cotrnk_{U/})^\op }
\RHom_{\Dmod(U)}
\bigt{
(j_{U \to U'})_*
(\DD_U \F_{U} ),
\omega_{U'}
}
\simeq
0.
$$

\sssec*{Step 6}

By adjunction (using cotruncativeness), we rewrite the LHS as
$$
\lim_{U' \in (\Cotrnk_{U/})^\op }
\RHom_{\Dmod(U)}
\bigt{
\DD_U \F_{U} ,
(j_{U \to U'})^? \omega_{U'}
}
$$
and further as
$$
\RHom_{\Dmod(U)}
\bigt{
\DD_U \F_{U} ,
\lim_{U' \in (\Cotrnk_{U/})^\op }
(j_{U \to U'})^? \omega_{U'}
}.
$$
Since $\DD_U$ is an involution on $\Dmod(U)^\cpt$, we deduce that 
$$
\lim_{U' \in (\Cotrnk_{U/})^\op }
(j_{U \to U'})^? \omega_{U'}
\simeq 0,
$$
which is what we were looking for.
\end{proof}

\sssec{}

As a corollary of the vanishing of $(p_{\Bun_G})_! \circ j_*$, we deduce that, for any $\F \in \Dmod(\Bun_G)$ and any $Z = \Bun_G - U$ with $U$ cotruncative, we have
$$
(p_{\Bun_G})_!(\F)
\simeq
(p_Z)_!(i_Z^!(\F)).
$$
This means that $\F$ and any of its \virg{tails} have the same cohomology with compact support. In particular, taking $\F = \omega_{\Bun_G}$ in the above formula and dualizing, we obtain that pullback in de Rham cohomology yields the isomorphism
$$
H^*_\dR(\Bun_G)
\simeq
H^*_\dR(\Bun_G -U).
$$

\sec{Proof of Theorem \ref{mainthm:DL factors thru QCoh} } \label{sec: DL spectral sends all to temp}

Since from now on we only consider the spectral side of geometric Langlands, let us switch $\Gch$ with $G$ and consider the endo-functor $\DL_G^\spec$ of $\ICoh_\N(\LSG)$.
First, we need to show that such functor annihilates the subcategory of $\ICoh_\N(\LSG)$ right orthogonal to $\QCoh(\LSG)$.
This will already imply that $\DL_G^\spec$ factors as
$$
\ICoh_\N(\LSG)
\xto{\Psi}
\QCoh(\LSG)
\to
\QCoh(\LSG)
\xto{\Xi}
\ICoh_\N(\LSG),
$$
where the middle arrow is the action by a D-module on $\LSG$.
Second, we will identify such D-module with the Steinberg D-module $\St_G$.

\ssec{Singular support and enhanced Eisenstein series} \label{ssec:prelims}

We assume familiarity with the theory of singular support for coherent sheaves on quasi-smooth stacks, see \cite{AG1} and \cite{AG2}.
We also assume some familiarity with the theory of $\H$, as developed in \cite{shvcatH} and in \cite{strong-gluing}. The latter two references are not strictly necessary for the proof, but they help streamline the argument.

\sssec{}

As the stack $\LSG$ is quasi-smooth, ind-coherent sheaves on it get assigned a singular support in $\Sing(\LSG)$. 
Recall that $\Sing(\LSG)$ parametrizes pairs $(\sigma,A)$ where $\sigma$ is a $G$-local system and $A$ a horizontal section of the flat vector bundle $\g^*_{\sigma}$. Let $\N \subset \Sing(\LSG)$ denote the \emph{global nilpotent cone}, that is, the closed conical locus defined by the requiring that $A$ be nilpotent.

\sssec{}

For a schematic map $f: \X \to \Y$ of quasi-smooth stacks, we denote by $\Y^\wedge_{\X}$ the formal completion of $f$ and by $\ICoh(\Y^\wedge_\X)$ the DG category of ind-coherent sheaves on it. Let $\wh f: \Y^\wedge_\X \to \Y$ the map induced by $f$.
For $\M \subset \Sing(\X)$ a closed conical subset, set $\ICoh_\M(\Y^\wedge_\X)$ to be the fiber product DG category
$$
\ICoh_\M(\Y^\wedge_\X)
:=
\ICoh(\Y^\wedge_\X)
\ustimes{\ICoh(\X)}
\ICoh_\M(\X).
$$
These notations all agree with the ones used in \cite{AG2}.

\sssec{} \label{sssec:propagation of sing supp}

We need to recall the \emph{rule of propagation of singular support} under pushforwards, see \cite[Section 7]{AG1}. Given $f:\X \to \Y$ as above, we have two natural maps 
$$
\Sing(\X)
\xleftarrow{\fs_f}
\X \times_\Y \Sing(\Y) \xto{\ft_f} \Sing(\Y)
$$
at the level of the spaces of singularities: $\fs_f$ is the \emph{singular codifferential}, while $\ft_f$ is simply the projection to the second component.
Now, let $\M \subseteq \Sing(\X)$ and $\N \subseteq \Sing(\Y)$ be closed conical subsets.

\begin{prop} \label{prop:propagation of sing supp}
With the above notation, assume further that $\ft_f \circ (\fs_f)^{-1}(\M) \subseteq \N$. Then:
\begin{itemize}
\item
$f_*^\ICoh: \ICoh(\X) \to \ICoh(\Y)$ restricts to a functor $f_*^\ICoh: \ICoh_\M(\X) \to \ICoh_\N(\Y)$;
\smallskip
\item
$(\wh f)_*^\ICoh: \ICoh(\Y^\wedge_\X) \to \ICoh(\Y)$ restricts to a functor 
$$
(\wh f)_*^\ICoh: \ICoh_\M(\Y^\wedge_\X) \to \ICoh_\N(\Y).
$$
\end{itemize} 
\end{prop}

\begin{proof}
The first item is \cite[Lemma 8.4.5]{AG1}. The second item is an immediate consequence: the essential image of the push-forward functor $\ICoh_\M(\X) \to \ICoh_\M(\Y^\wedge_\X)$ along $\X \to \Y^\wedge_\X$ generates the target under colimits.
\end{proof}

\sssec{} \label{sssec:Eis-enh definition}

The definition of the \emph{enhanced Eisenstein series} functor
$$
\Eis_P^{\enh, \spec}:
\ICoh_{\N_{P,M}}((\LSG)^\wedge_{\LSP})
\xto{(\wh\p_P)_*^\ICoh}
\ICoh_\N(\LSG)
$$
goes as follows:
\begin{itemize}

\item
by definition, the substack $\N_{P,M} \subseteq \Sing(\LSP)$ parametrizes pairs $(\sigma_P, A_M)$, where $\sigma_P$ is a $P$-local system and $A_M$ is a nilpotent horizontal section of $\fm^*_{\sigma_P}$; 
\item
as is standard, $\p_P: \LSP \to \LSG$ denotes the induction map, so the functor $(\wh\p_P)_*^\ICoh$ is the $\ICoh$-pushforward along the map $(\LSG)^\wedge_{\LSP} \to \LSG$;
\item
a simple application of the rule of propagation of singular support ensures that this functor does indeed send $\ICoh_{\N_{P,M}}((\LSG)^\wedge_{\LSP})$ to $\ICoh_\N(\LSG)$.
\end{itemize}

\sssec{} \label{sssec:CT-enh definition}

The enhanced constant term functor $\CT_P^{\enh, \spec}$ is, by definition, the right adjoint to $\Eis_P^{\enh, \spec}$. Tautologically, it can be expressed as the composition
$$
\CT_P^{\enh, \spec}:
\ICoh_\N(\LSG)
\hto
\ICoh(\LSG)
\xto{(\wh\p_P)^{!,\ICoh}}
\ICoh((\LSG)^\wedge_{\LSP})
\xto{\Psi}
\ICoh_{\N_{P,M}}((\LSG)^\wedge_{\LSP}),
$$
where the rightmost functor is the natural projection (right adjoint to the obvious inclusion).

\sssec{}

By adjunction, the assignment $P \squigto \Eis_P^{\enh, \spec} \circ \CT_P^{\enh, \spec}$ upgrades to a functor 
$$
\Par' \longto \Fun(\ICoh_\N(\LSG), \ICoh_\N(\LSG)).
$$
By adjunction again, we obtain a natural arrow
$$
\uscolim{P \in \Par'} \;
\Eis_P^{\enh, \spec} \circ \CT_P^{\enh, \spec}
\longto
\id_{\ICoh_\N(\LSG)}
$$
whose cone is by definition the functor $\DL_G^\spec$.

\ssec{Proof of Theorem \ref{mainthm:DL factors thru QCoh}}

The proof rests on a contractibility statement proven in \cite{AG2}, to which we reduce via a \virg{microlocal} argument as in \cite[Section 2.2-2.3]{strong-gluing}.

\sssec{}

By construction, for any $P \in \Par$, the functor $\Eis_P^{\enh, \spec} \circ \CT_P^{\enh, \spec}$ commutes with the action of $\H(\LS_G)$, so the same holds true for $\DL_G^\spec$. 
Hence, we expect the functor to be given by the action of an object $\F_\DL \in \Dmod(\N)^\unshift$: indeed, we conjecture that
$$
\End_{\H(\Y)} (\ICoh_\N(\Y))
\simeq
\Dmod(\N)^{\unshift}.
$$

\sssec{}

To work around this conjecture, we work on the smooth atlas $\LSG^{x} \tto \LSG$ obtained by choosing a point $x \in X$ and by considering $G$-local systems with a trivialization at $x$. By \cite[Section 10.6]{AG1}, we know that $\LSG^{x}$ is a global complete intersection scheme.

The reason this fact is useful is that, for $Y$ a global complete intersection scheme, we know that $\Dmod(\Sing(Y))^\Rightarrow$ acts on $\ICoh(Y)$, see \cite[Section 2.2]{strong-gluing}. Let us denote such action by $\ast$.

\sssec{A notational convention}

For $Z$ a space over $\LSG$, denote by $Z^x := Z \times_{\LSG} \LSG^x $ its pullback along the atlas. Similarly, for a map $f: Z_1 \to Z_2$ of spaces over $\LSG$, denote by $f^x: Z_1^x \to Z_2^x$ the base-changed map.
For instance, 
$$
\N^x = \LSG^x \times_{\LSG} \N,
\hspace{.4cm}
\LSP^x =  \LSP \ustimes{\LSG} \LSG^x,
$$
and
\begin{equation} \label{eqn:pP pulled back to LSGx}
\fp_P^x:
\LSP^x = \LSP \ustimes{\LSG} \LSG^x 
\longto \LSG^x
\end{equation}
is the natural induction map.

\sssec{}

Now consider the DG category $\ICoh_{\N^x}(\LSG^x)$, which is acted upon by $\Dmod(\N_x)^\Rightarrow$. By \virg{pulling-back} the constructions of Sections \ref{sssec:Eis-enh definition}-\ref{sssec:CT-enh definition} to $\LSG^x$, we obtain comonads $\Eis_P^{\enh, \spec, x} \circ \CT_P^{\enh, \spec, x}$. Indeed, observe that
the rule of propagation of singular support yields a functor
$$
(\wh {\fp_P^x})_*^\ICoh:
\ICoh_{\N^x_{P,M}}((\LSG^x)^\wedge_{\LSP^x}) \longto
\ICoh_{\N^x}(\LSG^x).
$$
Accordingly, we have a resulting functor $\DL_G^{\spec, x}: \ICoh_{\N^x}(\LSG^x) \to  \ICoh_{\N^x}(\LSG^x)$. Our current goal is to exhibit an object $\F^x_\DL \in \Dmod(\N^x)^\Rightarrow$ such that $\DL_G^{\spec, x} \simeq \F^x_\DL \ast -$.

\sssec{}

We fix a $G$-equivariant identification $\g^* \simeq \g$ once and for all, so that $A$ will be always regarded as a horizontal section of the adjoint bundle.

Consider the stack $\N_P \subseteq \LSP \times_{\LSG} \Sing(\LSG)$ consisting of pairs $(\sigma_P, A_P \in H^0_\dR(X, \p_{\sigma_P}))$ for which $A_P$ is nilpotent.
Note that the base-change of $\fp_P$ along $\Sing(\LSG) \to \LSG$ restricts to a map $\p_P^\Sing: \N_P \to \N$.

%%%$\N^x_P \subseteq \LSP \times_{\LSG} \Sing(\LSG^x) \simeq \LSP \times_{\LSG} \LSG^x \times_{\LSG} \Sing(\LSG)$ consists of triples $(\sigma_P, \gamma_G, A_P \in H^0_\dR(X, \p_{\sigma_P}))$ for which $A_P$ is nilpotent.

\begin{prop}
The comonad 
$$
\Eis_P^{\enh, \spec, x} \circ \CT_P^{\enh, \spec, x}:
\ICoh_{\N^x}(\LSG^x)
\longto
\ICoh_{\N^x}(\LSG^x)
$$ 
is given, up to shift of grading, by the object
$$
(\p_P^{\Sing,x})_{*,\dR} (\omega_{\N_P^x}).
$$
%%%%where:
%%%%\begin{itemize}
%%%%\item
%%%%$\N^x_P \subseteq \LSP \times_{\LSG} \Sing(\LSG^x) \simeq \LSP \times_{\LSG} \LSG^x \times_{\LSG} \Sing(\LSG)$ consists of triples $(\sigma_P, \gamma_G, A_P \in H^0_\dR(X, \p_{\sigma_P}))$ for which $A_P$ is nilpotent;
%%%%\item
%%%%$\p_P^{\Sing,x}: \N^x_P \to \N^x$ is the induction map determined by $P \subseteq G$ and $\p \subseteq \g$.
%%%%\end{itemize} 
\end{prop}

\begin{proof}
The comonad in question is the result of a general construction that takes the map \eqref{eqn:pP pulled back to LSGx} and the sets $\N^x_{P,M}$, $\N^x$ as inputs.
Since \eqref{eqn:pP pulled back to LSGx} is a proper map of quasi-smooth schemes, with the target a global complete intersection, we will be able to use the theory developed in \cite[Section 2]{strong-gluing}.

Here is the general paradigm that we will apply. Let $f: X \to Y$ be a proper map of quasi-smooth schemes, with $Y$ a global complete intersection. Let 
$$
M \subseteq \Sing(X),
\hspace{.4cm}
N \subseteq \Sing(Y)
$$ 
be closed conical subsets with the property that $\ft_f \circ \fs_f^{-1}(M) \subseteq N$. This assumption implies that the closed embedding $\fs_f^{-1}(M) \times_{\Sing(Y)} N \hto \fs_f^{-1}(M) $ is an isomorphism\footnote{when dealing with subsets of the stack of singularities, it is always understood that we work at the reduced level.}.
We need to compute the comonad of the adjunction
$$ %%%best-colocalization-diagram-prototype
\begin{tikzpicture}[scale=1.5]
\node (a) at (0,1) {$\ICoh_M(Y^\wedge_X) $};
\node (b) at (3,1) {$\ICoh_N(Y)$.};
\path[->,font=\scriptsize,>=angle 90]
([yshift= 1.5pt]a.east) edge node[above] {$\wh f_*^\ICoh$} ([yshift= 1.5pt]b.west);
\path[->,font=\scriptsize,>=angle 90]
([yshift= -1.5pt]b.west) edge node[below] { } ([yshift= -1.5pt]a.east);
\end{tikzpicture}
$$
Explicitly, this is given by the composition
$$
\ICoh_N(Y)
\hto
\ICoh(Y) \xto{\wh f^!}
\ICoh(Y^\wedge_X) 
\tto
\ICoh_M(Y^\wedge_X) 
\xto{\wh f_*^\ICoh} 
\ICoh_N(Y).
$$
We use \virg{microlocality} (i.e., the following equivalences, proven partly in \cite[Section 3]{AG2} and partly in \cite[Sections 2.2-2.3]{strong-gluing}) to write
\begin{eqnarray}
\ICoh_N(Y) 
\nonumber
& \simeq & 
\ICoh(Y)
\usotimes{\Dmod(\Sing(Y))^\Rightarrow} 
\Dmod(N)^\Rightarrow,
\\
\label{eqn:micro-1}
\ICoh(Y^\wedge_X) 
& \simeq & 
\ICoh(Y) \usotimes{\Dmod(\Sing(Y))^\Rightarrow} 
\Dmod(X \times_Y \Sing(Y))^\Rightarrow,
\\
\label{eqn:micro-2}
\ICoh_M(Y^\wedge_X) 
& \simeq & 
\ICoh(Y) \usotimes{\Dmod(\Sing(Y))^\Rightarrow} 
\Dmod(\fs_f^{-1}(M))^\Rightarrow.
\end{eqnarray}
Under these equivalences, the adjunction in question is tensored up (up to a shift of grading) from 
$$ %%%best-colocalization-diagram-prototype
\begin{tikzpicture}[scale=1.5]
\node (a) at (0,1) {$\Dmod \bigt{ \fs_f^{-1}(M)}$};
\node (b) at (3,1) {$\Dmod(N)$,};
\path[->,font=\scriptsize,>=angle 90]
([yshift= 1.5pt]a.east) edge node[above] {$\pi_{*,\dR}$} ([yshift= 1.5pt]b.west);
\path[->,font=\scriptsize,>=angle 90]
([yshift= -1.5pt]b.west) edge node[below] {$\pi^{!,\dR}$} ([yshift= -1.5pt]a.east);
\end{tikzpicture}
$$
where $\pi: \fs_f^{-1}(M) \simeq \fs_f^{-1}(M) \ustimes{\Sing(Y)} N \to N$ is the obvious (proper) projection.

Coming back to our case, we immediately\footnote{One needs to unravel the effect of the identification $\g^* \simeq \g$: under such identification, the $P$-representation $\g^* \times_{\fp^*} \fm^*$ corresponds to $\g \times_{\g/\fu} (\fp/\fu) \simeq \p$, the adjoint $P$-representation.} see that $\fs_f^{-1}(M) \simeq\N^x_P$, while $\pi$ is $\p_P^{\Sing,x}$.
\end{proof}

\sssec{}

Recall now that the equivalences \eqref{eqn:micro-1} and \eqref{eqn:micro-2} are naturally functorial in $X$: this can be established by paraphrasing the discussion of \cite[Section 2.3.5]{strong-gluing}.
In our particular case, for $Q \subseteq P$, the natural transformation
$$
\Eis_Q^{\enh, \spec, x} \circ \CT_Q^{\enh, \spec, x}
\to
\Eis_P^{\enh, \spec, x} \circ \CT_P^{\enh, \spec, x}
$$
is induced (under the result of the above proposition) by the natural map
$$
(\p_Q^{\Sing,x})_{*,\dR} (\omega_{\N^x_Q})
\longto
(\p_P^{\Sing,x})_{*,\dR} (\omega_{\N^x_P}).
$$

\sssec{}

Consequently, $\DL_G^{\spec,x}$ corresponds (up to shift of grading) to the object
\begin{equation} \label{eqn:map forming sing steinberg}
\F^x_{\DL}:=
\coker
\bigt{
\uscolim{P \in \Par'} \,
(\p_P^{\Sing,x})_{*,\dR} (\omega_{\N^x_P})
\longto
\omega_{\N^x}
} \in \Dmod(\N^x).
\end{equation}
Clearly, this object is the pullback along $\LSG^x \to \LSG$ of an object $\F_{\DL} \in \Dmod(\N)$, obtained by removing the decorations \virg{$x$} in the above formula.

\begin{prop}
When restricted to the complement of the zero section $\N^\circ := \N - {\LSG}$, the above object $\F_{\DL}$ vanishes.
\end{prop}

\begin{proof}
Let $(\sigma,A)$ be a geometric point of $\N$. According to the notation of \cite[Section 7.1.4]{AG2}, we have $\N_P \times_\N {(\sigma,A)} \simeq \Spr_P^{\sigma,A}$, the scheme of $P$-reductions of $\sigma$ with the property that $A \in H^0(X_\dR, \p_\sigma)$.
To prove the claim, it suffices to show that the $!$-fiber of $\F_{\DL}$ at any $(\sigma, A) \in \N^\circ$ is zero. By base change, this is equivalent to checking that 
$$
\Spr_\Glued^{\sigma, A} 
:= 
\uscolim{P \in \Par'} \; \Spr_P^{\sigma,A}
$$
is homologically contractible for any $A \neq 0$. This is exactly the statement of \cite[Theorem 7.2.5]{AG2}.
\end{proof}

\sssec{}
 It follows that $\DL_G^\spec$ annihilates the category of singularities
$$
\ICoh_\N(\LSG)^\circ := \ICoh_\N(\LSG)/\QCoh(\LSG).
$$
Thus, $\DL_G^\spec$ can be viewed as an endofunctor of $\QCoh(\LSG)$. Now, any endo-functor on $\QCoh(\LSG)$ that commutes with the $\H(\LSG)$-action must be given by a D-module of $\LSG$: this is the simplest application of \cite[Theorem 1.9.2]{shvcatH}. Such D-module is readily available: it is given by the formula
$$
\i^! 
\Bigt{
\coker 
\Bigt{
\uscolim{P \in \Par'} \,
(\p_P^{\Sing})_{*,\dR} (\omega_{\N_P})
\longto
\omega_\N
}
}
\in \Dmod(\LSG),
$$
where $\i: \LSG \hto \N$ is the inclusion of the zero section.
By base-change, the latter simplifies to the \emph{Steinberg object}:
$$
\St_G:=
\coker
\Bigt{
\uscolim{P \in \Par'} \,
(\p_P)_{*,\dR} (\omega_{\LSP}) 
\longto
\omega_{\LSG}
}
\in \Dmod(\LSG).
$$

\sec{Proof of Theorem \ref{mainthm: StG vs StM}} \label{sec:funct eqn St}

In this section, we use some Weyl combinatorics to prove the main property of $\St_G$, that is, Theorem \ref{mainthm: StG vs StM}.
Let us recall the statement: for $P_0$ a parabolic subgroup of $G$ with Levi $M_0$, we need to construct a canonical isomorphism 
\begin{equation} \label{eqn:St with P0 for proof}
\p_{P_0}^!(\St_G) \simeq \i_{P_0,*}(\St_{M_0})[\rk(G)-\rk(M_0)]
\end{equation}
in $\Dmod(\LS_{P_0})$, where $\i_{P_0}: \LS_{M_0} \to \LS_{P_0}$ is the natural map and $\rk$ denotes the semisimple rank of a reductive group.
We will later deduce Theorem \ref{mainthmbis: fibers of St} which describes the geometric fibers of $\St_G$.

\ssec{Preliminaries} \label{ssec:prelim}

We will use some standard Weyl combinatorics: we refer to \cite[Sections 8.1-8.2]{AG2} for a handy review and for the notation we use. In particular, $\sR$ (respectively, $\sR^+$) denotes the set of roots of $G$ (respectively, positive roots with respect to the chosen $B$). For a parabolic $P \subseteq G$, let $J_P$ be the subset of the Dynkin diagram associated to $P$ (e.g., $J_B = \emptyset$).

\sssec{}

In the proof that follows, we assume that $P_0$ is a proper standard parabolic. If $P_0$ is not standard, the strategy is the same, up to multiplying $w_0'$ by an appropriate element of $W$. Below, we abuse notation and write $J_0$ in place of $J_{P_0}$.

\sssec{}

Let $W' := \{w \in W \, : \,  w^{-1}(J_0) \subseteq \sR^+\}$. The quotient stack $P_0 \backslash G/P$ has strata 
indexed by $W'_P := \{ w \in W'  \, : \, w(J_P) \subseteq \sR^+ \}$.
For $w \in W'$ (but not necessarily in $W'_P$), the notations $(P_0 \backslash G/P)^{\leq w}$ and $(P_0 \backslash G/P)^{<w}$ have their evident meanings. 
We also set
$$
(P_0 \backslash G/P)^{ w} := 
(P_0 \backslash G/P)^{\leq w} -
(P_0 \backslash G/P)^{< w}
\simeq
\begin{cases}
P_0 \backslash P_0 w P/P
\simeq
\pt/(P_0 \cap w P w^{-1}) & \mbox{if $w \in W'_P$ } \\
\emptyset & \mbox{if $w \in W' - W'_P$}.
\end{cases}
$$

\sssec{} \label{sssec:about w0'}

Recall that $W'$ has a unique longest element $w_0'$, characterized by the fact that $w_0'(\sR^+) \cap \sR^+ = \sR_{J_0}^+$. Alternatively: $w_0'$ is the product $w_{0,P_0} \cdot w_0$, where $w_{0,P_0}$ and $w_0$ are the longest elements of $W_{M_0}$ and $W$ respectively. From this expression, it is clear that $(w_0')^{-1}$ sends the simple roots of $S_{P_0}$ to simple roots; we define $ K_{0} := (w_0')^{-1}(J_{P_0}) \subseteq S$.

\medskip

Consequently, 
$$
(P_0 \backslash G/P)^{ w_0'}
\simeq
\begin{cases}
\pt/(M_0 \cap P) & \mbox{if $J_P  \subseteq K_0$ } \\
\emptyset & \mbox{if $J_P \nsubseteq  K_0$}.
\end{cases}
$$

\sssec{}

Consider the mapping stack $\Y_P :=\MMaps(X_\dR, P_0 \backslash G/P)$ and its closed substacks 
$$
\Y_{P, \leq w} :=\MMaps(X_\dR, (P_0 \backslash G/P)^{\leq w}).
$$
Tautologically, we have:
$$
\Y_P
\simeq
\uscolim{w \in W'} \;
\Y_{P, \leq w}.
$$
Define also  $\Y_{P, < w}$ and $\Y_{P,w}$ is a similar way.
For instance, we have
$$
\Y_{P,w} 
\simeq 
\begin{cases}
\LS_{P^w \cap P_0} 
&
\mbox{ if $w(J_P) \subseteq \sR^+$} \\
\emptyset & \mbox{otherwise};
\end{cases}
$$
Denote by 
$$
\begin{array}{ll}
\pi_{P}: \Y_{P} \longto \LS_{P_0}
&
\pi_{P, \leq w}: \Y_{P, \leq w} \longto \LS_{P_0} \\
\pi_{P, < w}: \Y_{P, <w} \longto \LS_{P_0}
&
\pi_{P,  w}: \Y_{P, w} \longto \LS_{P_0}
\end{array}
$$
the obvious maps. 

\begin{example}
We have seen above that $\Y_{P,w}  \simeq \LS_{P^w \cap P_0}$ whenever $w(J_P) \subseteq \sR^+$. In this case, $\pi_{P,w}$ is the induction map $\i_{P^w \cap P_0 \to P_0}$.
\end{example}

\ssec{The proof}

We are now ready to construct the natural isomorphism appearing in \eqref{eqn:St with P0 for proof}.

\sssec{}

Set $\S := \p_{P_0}^!(\St_G) \in\Dmod(\LS_{P_0})$.  Obviously, 
$$
\S
\simeq
\coker
\Bigt{
\uscolim{P \in \Par'} \; (\pi_{P})_* \omega_{\Y_{P}}
\longto
\omega_{\LS_{P_0}}
}.
$$
Hence, $\S \simeq \colim_{w \in W'} \S^{\leq w}$, where
$$
\S^{\leq w}
:=
\coker
\Bigt{
\uscolim{P \in \Par'} \; (\pi_{P, \leq w})_* \omega_{\Y_{P, \leq w}}
\longto
\omega_{\LS_{P_0}}
}.
$$

\begin{lem}
The object $\S^{\leq 1} \in \Dmod(\LS_{P_0})$ is isomorphic to the zero object.
\end{lem}

\begin{proof}
Since $(P_0 \backslash G/P)^{\leq 1} = \pt/(P_0 \cap P)$ for any $P \in \Par'$, we obtain
$$
\uscolim{P \in \Par'} \; (\pi_{P, \leq 1})_* \omega_{\Y_{P, \leq 1}}
\simeq
\uscolim{P \in \Par'} \; (\i_{P_0 \cap P \to P_0})_* \omega_{\LS_{P_0 \cap P}} 
\simeq
\uscolim{P \subseteq P_0} \; (\i_{P_0 \cap P \to P_0})_* \omega_{\LS_{P_0 \cap P}}
\simeq
\omega_{\LS_{P_0}}.
$$
Then the assertion is clear.
\end{proof}

\begin{lem}
For any $w \in W' - \{1,w_0'\}$, the natural map $\S^{<w} \to \S^{\leq w}$ is an isomorphism. 
\end{lem}

\begin{proof}
It suffices to show that the map
$$
\uscolim{P \in \Par'} \; (\pi_{P, <w})_* \omega_{\Y_{P, <w}}
\longto
\uscolim{P \in \Par'} \; (\pi_{P, \leq w})_* \omega_{\Y_{P, \leq w}}
$$
is an isomorphism in $\Dmod(\LS_{P_0})$. As argued in \cite[Lemma 6.1.7]{AG2}, this can be checked at the level of geometric points, that is, after pulling back to a $P_0$-local system $\sigma_{P_0} \to \LS_{P_0}$. Observe that 
$$
\Y_{P, <w} \ustimes{\LS_{P_0}} \sigma_{P_0}
\simeq
\Spr_P^{\sigma, < w}
\hspace{.6cm}
\Y_{P, \leq w} \ustimes{\LS_{P_0}} \sigma_{P_0}
\simeq
\Spr_P^{\sigma, \leq w}
$$
in the notation of \cite[Section 7.1.4]{AG2} and \cite{strong-gluing}. Hence, we just need to show that the map
$$
H_* (\Spr_{\Glued}^{\sigma, < w})
\longto
H_* (\Spr_{\Glued}^{\sigma, \leq w})
$$
is an isomorphism of complexes of vector spaces. Equivalently, we need to show that the prestack
$$
\Spr_{\Glued}^{\sigma, \leq w}/ \Spr_{\Glued}^{\sigma, < w}
$$
is homologically contractible.
The proof is the special case of \cite[Section 8.5.1-8.5.5]{AG2} for $A=0$. (Compare with \cite[Remark 8.3.2]{AG2}.)
\end{proof}

\sssec{}

The two lemmas above imply that $\S^{< w_0'} \simeq 0$: indeed, $\S^{< w_0'} \simeq \colim_{u < w_0'} \S^{\leq u}$. 
Hence, 
$$
\S 
\simeq 
\S^{\leq w_0'} 
\simeq 
\coker (\S^{< w_0'} \to \S^{\leq w_0'}).
$$
On the other hand, we tautologically have
$$
\coker (\S^{< w_0'} \to \S^{\leq w_0'})
\simeq
\uscolim{P \in\Par'}
\coker
\Bigt{
(\pi_{P, < w_0'})_* \omega_{\Y_{P,< w_0'}}
\longto
(\pi_{P, \leq w_0'})_* \omega_{\Y_{P,\leq w_0'}}
}[1].
$$
Since $\Y_{P,w_0'} \simeq  \Y_{P,\leq w_0'} - \Y_{P,< w_0'}$, the open-closed fiber sequence, combined with the discussion of Section \ref{ssec:prelim}, yields
$$
\F_{P} :=\coker
\Bigt{
(\pi_{P, < w_0'})_* \omega_{\Y_{P,< w_0'}}
\longto
(\pi_{P, \leq w_0'})_* \omega_{\Y_{P,\leq w_0'}}
}
\simeq 
\begin{cases}
0 & \mbox{if $P \nsubseteq P_{K_0}$} \\
(\i_{P \cap M_0 \to P_0})_* \omega_{\LS_{P \cap M_0}}
& 
\mbox{if $P \subseteq P_{K_0}$}.
\end{cases}
$$

\sssec{}

To conclude our proof, it remains to simplify the RHS of 
$$
\S \simeq \bigt{ \uscolim{P \in \Par'} \; \F_P }[1]
$$
by showing that
$$
\uscolim{P \in \Par'} \; \F_P \simeq (\i_{M_0 \to P_0})_*(\St_{M_0})[\rk(G) - \rk(M_0)-1].
$$
The proof of this latter claim amounts to applying the following general lemma to the functor  
$$
\F_\bullet: \Par' \to \Dmod(\LS_{P_0}).
$$

\begin{lem} \label{lem:colim of poset of parts }

For a finite set $A$, denote by $\P(A)$ the poset of parts of $A$ and by $\P'(A) := \P(A) - \{A\}$ the poset of proper parts of $A$. 

Let $A \subsetneq B$ be two finite sets and $\phi: \P'(B) \to \C$ a functor to a DG category $\C$.
If $\phi(J) = 0$ for any $J \nsubseteq A$, then 
$$
\colim \phi
\simeq
\coker
\Bigt{
\colim \restr \phi {\P'(A)} \to \phi(A)
}
\big[ |B-A|-1 \big ]. 
$$
\end{lem}

\begin{proof}
Clearly, treating the case of $|A| = |B|-1$ is enough. Let $x \in B-A$ be the only extra element. 
The decomposition $\P'(B) = \P(A) \sqcup_{\P'(A)} (\P'(B)-\{A\})$
shows that the square
$$ 
\begin{tikzpicture}[scale=1.5]
\node (01) at (0,1) {$\colim \restr \phi {\P'(A)}$};
\node (11) at (2,1) {$\colim \restr \phi {\P(A)}$};
\node (10) at (2,0) {$\colim \phi$};
\node (00) at (0,0) {$\colim \restr \phi {(\P'(B)-\{A\})}$};
\path[->,font=\scriptsize,>=angle 90]
(00.east) edge node[above] {$ $} (10.west);
\path[->,font=\scriptsize,>=angle 90]
(01.south) edge node[left] {$ $}  (00.north);
\path[->,font=\scriptsize,>=angle 90]
(01.east) edge node[above] {$ $} (11.west);
\path[->,font=\scriptsize,>=angle 90]
(11.south) edge node[right] {$ $} node[left] {} (10.north);
\end{tikzpicture}
$$
is a pushout. Since $\P(A)$ has a final object ($A$ itself), it remains to show that the colimit of the restriction of $\phi$ to $(\P'(B)-\{A\})$ is zero.
Since the inclusion
$$
\P'(B)_{x/}
\longto
\P'(B)-A
$$
is cofinal, we have
$$
\colim \restr \phi {(\P'(B)-\{A\})}
\simeq
\colim \restr \phi {\P'(B)_{x/}}
$$
and the RHS is zero (indeed, $\phi$ is identically zero on $\P'(B)_{x/}$).
\end{proof}

\ssec{Proof of Theorem \ref{mainthmbis: fibers of St}}

Let us deduce Theorem \ref{mainthmbis: fibers of St} from  Theorem \ref{mainthm: StG vs StM}.
We use the following corollary as the main ingredient.

\begin{cor} \label{cor:values of St on sigmaM}
If $\sigma \simeq \sigma_M \times^M G$, then 
$$
\restr {\St_G} \sigma
\simeq
\restr {\St_M} {\sigma_M}[2 \cdot h^0(X_\dR, \fu_{\sigma_M}) + \rk(G)-\rk(M)],
$$
where $U$ in the unipotent radical of a parabolic with Levi $M$ and $\fu = Lie(U)$.
\end{cor}

\begin{proof}
The map $\sigma: \pt \to \LSG$ factors as $\pt \xto{\sigma_P} \LSP \xto{\p_P} \LSG$, where $\sigma_P$ is the $P$-local system induced by $\sigma_M$. Then base change yields
$$
\restr {\St_G} \sigma
\simeq
\Gamma_{\dR}(\omega_Y) \otimes \restr{\St_M}{\sigma_M}[\rk(G) - \rk(M)],
$$
where 
$$
Y := 
\{\sigma_M \times^M P \} \times_{\LSP} \LSM 
\simeq
\ul\Sect (X_\dR, U_{\sigma_M})
$$
is the DG scheme of $M$-reductions of $\sigma_M \times^M P$. The classical scheme underlying $Y$ is isomorphic to the vector space $H^0(X_\dR, \fu_{\sigma_M})$.
In particular, $Y^{\cl}$ is homologically contractible and smooth of dimension $h^0(X_\dR, \fu_{\sigma_M})$. The assertion follows.
\end{proof}

\sssec{}

If $\sigma \simeq \sigma_M \times^M G \in \LSG(\kk)$ with $\sigma_M$ irreducible, then Corollary \ref{cor:values of St on sigmaM} shows that 
$$
\restr {\St_G} {\sigma}
\simeq \kk [2 \cdot h^0(X_\dR, \fu_{\sigma_M}) + \rk(G)-\rk(M)].
$$
Viceversa, suppose that $\sigma$ is not semisimple: this means that $\sigma \simeq \sigma_P \times^P G$ for some $P \in \Par'$ and some $P$-local system $\sigma_P$ which is not $M$-reducible. Then $\restr {\St_G} {\sigma} = 0$ by Theorem \ref{mainthm: StG vs StM}.

\sec{Proof of Theorem \ref{mainthm:St-fully faithful on CohN}} \label{sec:St fully faith on CohN}

Consider the functor
$$
\ul\St_G \otimes - :
\QCoh(\LSG)
\longto
\QCoh(\LSG).
$$
In this section, we will prove Theorem \ref{mainthm:St-fully faithful on CohN}, which states that such functor is fully faithful when restricted to $\Coh_\N(\LSG)$.
As a key tool, we apply the second adjunction (an instance of Braden's theorem) in the context of $\Dmod(\LSG)$.

\ssec{Braden's theorem and contraction principle for local systems} \label{ssec:Braden}

In this section, we render some of the material of \cite{braden}, \cite{CT}, \cite{DG-Braden} to the setting of $G$-local systems.

\sssec{}

Consider the Eisenstein series functor 
$$
\Eis_{P,*}^{\Dmod}: \Dmod(\LSM) \longto \Dmod(\LSG)
$$ 
defined by $(\p_{P})_{*,\dR}\circ (\q_P)^{!,\dR}$. 
Note that de Rham push-forward $(\p_{P})_{*,\dR}$ is continuous since the map $\p_P$ is schematic. Our goal is to prove that $\Eis_{P,*}^{\Dmod}$ admits a left adjoint. Such left adjoint is at least partially defined: it is given by the formula
$$
\CT_{P,!}^{\Dmod} :=  (\q_P)_! \circ \p_P^{*,\dR}
:
 \Dmod(\LSG) 
 \longto 
 \Dmod(\LSM).
$$
The question is then to show that this functor is defined on the entire category $\Dmod(\LSG)$.

\sssec{}

Consider the functor dual to $\Eis_{P,*}^{\Dmod}$: namely, the constant term functor
$$
\CT_{P,*}^{\Dmod} :=  (\q_P)_{*,\dR} \circ (\p_P)^{!,\dR}
:
 \Dmod(\LSG) 
 \longto 
 \Dmod(\LSM).
$$
The push-forward $(\q_P)_{*,\dR}$ is continuous because the map $\q_P$ is \emph{safe} in the terminology of \cite{finiteness}.

\begin{thm} [Second adjunction] \label{thm:CT for Dmod on LS}
There is a natural isomorphism of functors: $\CT_{P,!}^{\Dmod} \simeq \CT_{P^-,*}^{\Dmod}$. In particular, $\CT_{P,!}^{\Dmod}$ is well defined on the entire $\Dmod(\LS_G)$.
\end{thm}

\begin{proof}
The proof is an instance of Braden's theorem. For instance, one might copy the one given in \cite{CT} for $G$-bundles.
\end{proof}

\sssec{}

Let us also record the following consequence of the contraction principle. For an appropriate cocharacter $\gamma: \Gm \to Z(M)$, the resulting $\Gm$-action on $\LSP$ is contracting (and trivializable), with fixed locus $\LSM$. This implies that $(\i_P)_{*,\dR}$ is fully faithful, with left adjoint isomorphic to $(\i_P)^{*,\dR} \simeq (\q_P)_{*,\dR}$.
Similarly, $(\q_P)^{!,\dR}$ is fully faithful, with left adjoint isomorphic to $(\i_P)^{!, \dR}$.
For the proofs, see \cite[Section 4.1.6]{CT}.

\ssec{D-module functoriality}

This is a quick reminder of the basic D-module functors on QCA algebraic stacks.
Recall the conventions of Section \ref{ssec: notation}.

\sssec{}

We denote by $(\ind_R, \oblv_R)$ the induction/forgetul functors for right D-modules. Recall that $\ind_R$ is dual (as well as left adjoint) to $\oblv_R$, with respect to the standard self dualities of $\Dmod(\Y)$ and $\ICoh(\Y)$.

The forgetful functor $\oblv_R$ intertwines the two types of $!$-pullbacks. By duality, $\ind_R$ intertwines $\ICoh$-pushforwards with renormalized de Rham push-forwards, see \cite{finiteness}. 

\sssec{}

We also have the induction/forgetful adjunction $(\ind_L, \oblv_L)$ for left D-modules. This adjunction is valid only for bounded (that is, eventually coconnective) stacks; we are not in danger, as we will only apply it to quasi-smooth stacks. The forgetful functor $\oblv_L$ intertwines $*$-pullbacks of quasi-coherent sheaves with $!$-pullbacks of D-modules.

\sssec{}

It remains to discuss the interaction between $\ind_L$ and $(\QCoh,*)$-pushforwards.
First off, we have $\oblv_R \simeq \Upsilon \circ \oblv_L$ and $\ind_L \simeq \ind_R \circ \Upsilon$.
Thus, the dual of $\ind_L$ is 
$$
(\ind_L)^\vee
\simeq
\Psi \circ \oblv_R.
$$
For $\Y$ a Gorenstein (for example, quasi-smooth) stack, we write $\L_\Y$ for the shifted line bundle $\Psi(\omega_\Y) \in \QCoh(\Y)$.
Abusing notation, for $H$ an affine algebraic group, we set $\L_H := \L_{\LS_H}$.

\begin{lem}
Let $f: \Y \to \Z$ be a map between Gorenstein QCA stacks.
Then
\begin{equation} \label{magic-formula}
\ind_L \, f_*
\simeq
f_{*,\ren} \, \ind_R \, \Xi_{\Y} ( f^*(\L_\Z) \otimes  - ),
\end{equation}
\begin{equation} \label{magic-formula-2}
\ind_L \, f_*
\simeq
f_{*,\ren} \, \ind_L \bigt{
 \L_\Y^{-1} \otimes f^*(\L_\Z) \otimes  - 
 }.
\end{equation}
\end{lem}

\begin{proof}
To check the first formula, let us pass to dual functors on both sides: we need to establish a functorial isomorphism
$$
f^* \circ \Psi_\Z \circ \oblv_R
\simeq
 f^*(\L_\Z) \otimes 
\Phi_\Y
\circ
\oblv_R
\circ 
f^{!,\dR},
$$
or equivalently (thanks to $\oblv_R = \Upsilon \oblv_L$), 
$$
f^* \circ \Psi_\Z \circ \Upsilon_\Z \circ \oblv_L
\simeq
 f^*(\L_\Z) \otimes 
\oblv_L
\circ 
f^{!,\dR}.
$$
The assertion is now manifest, as $\Psi_\Z \Upsilon_\Z = \L_\Z \otimes -$.
The second formula is proven in exactly the same way.
\end{proof}

\begin{cor}
Let $f: \Y \to \Z$ be a proper (in particular, schematic) map between Gorenstein QCA stacks. Then, for $Q \in \QCoh(\Y)$ and $\F \in \Dmod(\Z)$, there is a natural isomorphism
\begin{equation} \label{adj:hybrid}
\RHom_{\QCoh(\Z)}
\bigt{
f_*Q, \ul{\F}
}
\simeq
\RHom_{\QCoh(\Y)}
\Bigt{
f^*(\L_\Z) \otimes Q, \L_\Y \otimes \ul{f^!\F}
}.
\end{equation}
\end{cor}

\ssec{Setting up the proof}

\sssec{}

It will be actually convenient to slightly reformulate the result. 
Let us introduce the following terminology: we say that a functor $F: \C \to \D$ is fully faithful on a pair $(c,c') \in \C \times \C$ iff it induces an isomorphism 
$$
\RHom_\C(c, c') \xto{\;\; \simeq \;\;} \RHom_\D(F(c), F(c')).
$$

\sssec{}

It is clear that following theorem implies (and in fact it is equivalent to) Theorem \ref{mainthm:St-fully faithful on CohN}.

\begin{thm} \label{thm:St is fully faith on cpts}
The functor 
$$
\ul\St_G \otimes -
: 
\QCoh(\LSG)
\longto
\QCoh(\LSG)
$$
is fully faithful on pairs of the form $(c, c') \in \QCoh(\LSG) \times \Coh_\N(\LSG)$.
\end{thm}

We will prove this theorem by induction on the semisimple rank of $G$. 
For $T$, the assertion is obvious: this is the base of the induction. We henceforth assume that the theorem is true for any proper Levi subgroup of $G$. 

\sssec{}

Observe that the property of a continuous functor $F$ to be fully faithful on a pair $(c,c')$ is preserved by taking arbitrary colimits in the first variable, and Karoubi colimits (that is, finite colimits and retracts) in the second variable.
Hence, it is enough to show that $(c,c')$ has the required property for $c'$ running through a fixed set of Karoubi generators of $\Coh_\N(\LSG)$.

\sssec{}

Thanks to \cite[Corollary 13.3.10]{AG1}, we know that the objects
$$
(\p_P)_* (\F_P),
\hspace{.4cm}
\mbox{ for all
$P \in \Par$ and $\F_P \in \Perf(\LSP)$,}
$$ 
Karoubi-generate $\Coh_\N(\LSG)$. Thus, we need to show that the map
$$
\RHom(\F, \F')
\longto
\RHom(\F \otimes \ul\St_G, \F' \otimes \ul\St_G)
$$
is an isomorphism for $\F'$ as above and $\F$ arbitrary.
Let us distinguish two cases: $P \neq G$ (to be treated next, in Section \ref{ssec:first case}) and $P =G$ (to be treated later, in Section \ref{ssec:second case}).

\ssec{The first case: $P \neq G$} \label{ssec:first case}
 
 \sssec{}
 
Let $P$ be a proper parabolic.
We need to show that, for $\F \in \QCoh(\LSG)$ and $\F_P \in \Perf(\LSP)$, the natural map
$$
\RHom_{\QCoh(\LSG)}
\bigt{ 
\F, (\p_P)_*(\F_P)
}
\longto
\RHom_{\QCoh(\LSG)}
\bigt{
\F \otimes \ul\St_G, (\p_P)_*(\F_P) \otimes \ul\St_G
}
$$ 
is an isomorphism. By adjunction, we have:
$$
\RHom_{\QCoh(\LSG)}
\bigt{
\F \otimes \ul\St_G, (\p_P)_*(\F_P) \otimes \ul\St_G
}
\simeq
\RHom_{\QCoh(\LSP)}
\bigt{
(\p_P)^*(\F) \otimes \ul{(\i_P)_* \St_M}, \F_P \otimes \ul{(\i_P)_* \St_M}
}.
$$ 
Thus, the assertion reduces to the following one.

\begin{thm} \label{thm: P-fully faithfulness }
The functor 
$$
\ul{ (i_P)_* \St_M } \otimes -:
\Perf(\LSP) \to \QCoh(\LSP)
$$
is fully faithful.
\end{thm}

\begin{proof}
It suffices to prove that the map
\begin{equation} \label{eqn:StP fully faithful}
\RHom_{\QCoh(\LSP)} 
(\F, \O_{\LSP})
\longto
\RHom_{\ICoh(\LSP)}
\Bigt{
\F
\stackrel \act\otimes 
\oblv_R ( (i_P)_* \St_M )
,
\oblv_R ( (i_P)_* \St_M )
}
\end{equation}
is an isomorphism for any $\F \in \Perf(\LSP)$, where $\stackrel \act \otimes$ denotes the action of $\QCoh$ on $\ICoh$.

\sssec*{Step 1}

Let us start manipulating the RHS. By adjunction and then projection formula, it is isomorphic to 
$$
\RHom_{\Dmod(\LSP)}
\Bigt{
\ind_R
(\Upsilon\F)
\sotimes 
 (i_P)_* \St_M 
,
 (i_P)_* \St_M 
}
\simeq
\RHom_{\Dmod(\LSP)}
\Bigt{
 (i_P)_*
 \bigt{
i_P^!(\ind_R
(\Upsilon\F))
\sotimes 
 \St_M 
 }
,
 (i_P)_* \St_M 
}.
$$
Let us now recall that, by the contraction principle, the functor $(i_P)_*$ is fully faithful. Hence, the RHS of \eqref{eqn:StP fully faithful} is isomorphic to 
$$
\RHom_{\Dmod(\LSM)}
\Bigt{
i_P^!(\ind_R
(\Upsilon\F))
\sotimes 
 \St_M 
,
\St_M 
}.
$$

\sssec*{Step 2}

Our next goal is to eliminate the two occurrencies of $\St_M$ from the Hom space above. This will be done by a diagram chase, together with the induction hypothesis.
Consider the following cartesian square:
$$ 
\begin{tikzpicture}[scale=1.5]
\node (01) at (0,1) {$(\LSP)^\wedge_{\LSM}$};
\node (mezzo) at (1.5,1) {$\LSM$};
\node (11) at (3,1) {$(\LSM)_\dR$};
\node (10) at (3,0) {$(\LSP)_\dR$.};
\node (00) at (0,0) {$\LSP$};
\path[->,font=\scriptsize,>=angle 90]
(00.east) edge node[above] {$ $} (10.west);
\path[->,font=\scriptsize,>=angle 90]
(01.east) edge node[above] {$\xi$} (mezzo.west);
\path[->,font=\scriptsize,>=angle 90]
(mezzo.east) edge node[above] {$ $} (11.west);
\path[->,font=\scriptsize,>=angle 90]
(01.south) edge node[left] {$\wh i_P$}  (00.north);
\path[->,font=\scriptsize,>=angle 90]
(11.south) edge node[right] {$i_P$} node[left] {} (10.north);
\end{tikzpicture}
$$
%%%%%
%%%
%%%
Base-change along this diagram, together with the $(\ind_R, \oblv_R)$ adjunction, yields
$$
\RHom_{\Dmod(\LSM)}
\Bigt{
i_P^!(\ind_R
(\Upsilon\F))
\sotimes 
 \St_M 
,
\St_M 
}
\simeq
\RHom_{\ICoh(\LSM)}
\Bigt{
\xi_*^\ICoh \bigt{ \wh i_P^! (\Upsilon\F) }
\sotimes 
\oblv_R(\St_M)
,
\oblv_R(\St_M)
}.
$$
The two ind-coherent sheaves appearing on the RHS belong to the full subcategory $\Upsilon(\QCoh(\LSM))$: this is obvious for the right one; as for the left one, it suffices to notice that $\xi_*^\ICoh$ sends $\QCoh((\LSP)^\wedge_{\LSM})$ to $\QCoh(\LSM)$ since $\q_P$ is quasi-smooth.
Hence, we can use the induction hypothesis (that is, Theorem \ref{thm:St is fully faith on cpts} for the group $M$) to obtain
$$
\RHom_{\ICoh(\LSM)}
\Bigt{
\xi_*^\ICoh \bigt{ \wh i_P^! (\Upsilon\F) }
\sotimes 
\oblv_R(\St_M)
,
\oblv_R(\St_M)
}
\simeq
\RHom_{\ICoh(\LSM)}
\Bigt{
\xi_*^\ICoh \bigt{ \wh i_P^! (\Upsilon\F) }
,
\omega_{\LSM}
},
$$
which is in turn isomorphic to 
$$
\RHom_{\Dmod(\LSM)}
\Bigt{
i_P^!(\ind_R
(\Upsilon\F))
,
\omega_{\LSM}
}
$$
by reasoning backwards.

\sssec*{Step 3}

Recall that, by the contraction principle again, the functor $(q_P)_!:\Dmod(\LSP) \to \Dmod(\LSM)$ is well-defined and isomorphic to $i_P^{!}$. 
We conclude that
$$
\RHom_{\Dmod(\LSM)}
\Bigt{
i_P^!(\ind_R
(\Upsilon\F))
,
\omega_{\LSM}
}
\simeq
\RHom_{\Dmod(\LSP)}
\Bigt{
\ind_R
(\Upsilon\F)
,
\omega_{\LSP}
}.
$$
The RHS is now manifestly isomorphic to $\RHom_{\QCoh(\LSP)} 
(\F, \O_{\LSP})$, as desired.
\end{proof}

\ssec{The second case: $P = G$} \label{ssec:second case}

\sssec{}

The next case is the one with $P=G$, so that $\F'$ is perfect (while $\F$ is still arbitrary). We need to show that the map
$$
\RHom_{\QCoh(\LSG)}
\bigt{ \F, \F'}
\longto
\RHom_{\QCoh(\LSG)}
\bigt{
\ul\St_G \otimes \F, \ul\St_G \otimes \F'
}
$$
is an isomorphism.

\sssec{}

Without loss of generality, we may assume that $\F' \simeq \O_{\LSG}$. Thus, we need to prove that the arrow
\begin{equation} \label{eqn:wtSt-flly faithful case (any,G)}
\RHom_{\QCoh(\LSG)}
(\F, \O_{\LSG})
\longto
\RHom_{\QCoh(\LSG)}
(\F \otimes \ul\St_G, \ul\St_G)
\end{equation}
is an isomorphism for arbitrary $\F$. It suffices to do this for $\F$ running through a fixed collection of generators of $\QCoh(\LSG)$.
Thus we assume that:
\begin{itemize}
\item
either $\F = j_*(\F_0)$, with $j: \LSG^\irred \hto \LSG$ the open substack of irreducible $G$-local systems and $\F_0 \in \QCoh(\LSG^\irred)$,
\smallskip
 \item
or $\F = (\p_P)_*(\F_P)$ with $P \in \Par'$ and $\F_P \in \Perf(\LSP)$.
\end{itemize}
 We treat these two subcases separately.

\sssec{}

Let $\F = j_*(\F_0)$ for some $\F_0 \in \QCoh(\LSG^\irred)$. Note that $j_*(\F_0) \otimes \ul\St_G \simeq j_*(\F_0)$ by the projection formula. Hence, we just need to show that the map
$$
\RHom_{\QCoh(\LSG)}(j_*(\F_0), \O_{\LSG})
\longto
\RHom_{\QCoh(\LSG)}
\bigt{
j_*(\F_0), \ul{\St_G}
}
$$ 
is an isomorphism. Equivalently, we need to show that 
$$
\RHom_{\QCoh(\LSG)}
\Bigt{
j_*(\F_0), 
\uscolim{P \in \Par'} \;  \ul { \p_{P,*} \omega_{\LSP} 
}
}
\simeq 
0.
$$
This fact is a consequence of the next lemma.

\begin{lem} 
For any $P \in \Par'$ and any $\F \in \QCoh(\LSG^\irred)$, we have
$$
\RHom_{\QCoh(\LSG)}
\Bigt{
j_*(\F), \ul{(\p_P)_* \omega_{\LSP}} 
}
\simeq 
0.
$$
\end{lem} 
 
 \begin{proof}
Consider the functor $\Eis_{P,*}^{\Dmod}: \Dmod(\LSM) \to \Dmod(\LSG)$ defined by $\p_{P,*}\circ \q_P^!$.
Adjunction, together with \eqref{magic-formula-2}, gives
$$
\RHom_{\QCoh(\LSG)}
\bigt{
j_*(\F), \ul{(\p_P)_* \omega_{\LSP}} 
}
\simeq 
\RHom_{\Dmod(\LSG)}(j_*(\ind_L(\F)), (\p_P)_* \omega_{\LSP}).
$$
Then we need to show that any object of $\Dmod(\LSG^\irred)$ is left orthogonal to $\Eis_{P,*}^{\Dmod}(\omega_{\LSM}) \simeq (\p_P)_* \omega_{\LSP}$.
This follows immediately from the \virg{second adjunction}, that is, Theorem \ref{thm:CT for Dmod on LS}.
\end{proof}

\sssec{}

Finally, let us assume that $\F = (\p_P)_*(\F_P)$ in \eqref{eqn:wtSt-flly faithful case (any,G)}.
We need to show:

\begin{prop}
For any $\F_P \in \Perf(\LSP)$, the functor $\ul\St_G \otimes -$ yields an isomorphism
$$
\RHom_{\QCoh(\LSG)}
\bigt{
(\p_P)_* (\F_P), \O_{\LSG}
}
\simeq
\RHom_{\QCoh(\LSG)}
\bigt{
(\p_P)_* (\F_P) \otimes \ul{\St_G} , 
 \ul{\St_G}
}.
$$
\end{prop}

\begin{proof}
By adjunction, this is equivalent to checking that $\ul\St_G \otimes -$ yields an isomorphism
$$
\RHom_{\Dmod(\LSG)}
\Bigt{
\ind_L
\bigt{
(\p_P)_* \F_P
}, \omega_{\LSG}
}
\simeq
\RHom_{\Dmod(\LSG)}
\Bigt{
\ind_L \bigt{
(\p_P)_* (\F_P) \otimes \ul{\St_G}
}, 
\St_G
}.
$$
Thanks to \eqref{magic-formula}, which in our case looks like
$$
\ind_L \circ (\p_P)_*
\simeq
(\p_P)_{*,\dR} \, \ind_R \, \Xi_{\LSP} ( \p_P^*(\L_G) \otimes  - ),
$$
the LHS becomes 
\begin{eqnarray}
\nonumber
\RHom_{\Dmod(\LSG)}
\Bigt{
(\p_P)_{*,\dR} \, \ind_R \, \Xi_{\LSP} ( \p_P^*(\L_G) \otimes  \F_P )
,
\omega_{\LSG}
}
& \simeq &
\RHom_{\Dmod(\LSP)}
\Bigt{
\ind_R \, \Xi_{\LSP} ( \p_P^*(\L_G) \otimes  \F_P )
,
\omega_{\LSP}
}
\\
\nonumber
& \simeq &
\RHom_{\QCoh(\LSP)}
\Bigt{
\p_P^*(\L_G) \otimes  \F_P 
,
\L_{P}
}.
\end{eqnarray}
Similarly, the RHS side becomes
$$
\RHom_{\QCoh(\LSP)}
\Bigt{
\p_P^*(\L_G) \otimes  \F_P 
\otimes \ul{ (\i_P)_* \St_M }
,
\L_{P} 
\otimes 
\ul{ (\i_P)_* \St_M }
}.
$$
Then we are back to the statement of Theorem \ref{thm: P-fully faithfulness }.
\end{proof}

\sec{Proof of Theorem \ref{mainthm:principal monoidal ideal}} \label{sec:divisibilty by St}

We wish to show that $\ul\St_G$ and $\St_G$ are generators of the monoidal ideals $\QCoh(\LSG)^\ss \subseteq \QCoh(\LSG)$ and $\Dmod(\LSG)^\ss \subseteq \Dmod(\LSG)$, respectively. The claim for $\ul\St_G$, treated in Section \ref{ssec:qcoh case}, will follow tautologically from the claim for $\St_G$, treated immediately below.

\ssec{The D-module case}

\sssec{}

Let us recall the definition of $\Dmod(\LSG)^\ss$. An object $\F \in \Dmod(\LSG)$ belongs to $\Dmod(\LSG)^\ss$ iff
$$
(\p_P)^{!,\dR} (\F) 
\in
(\i_P)_{*,\dR} \bigt{ \Dmod(\LSM) }
,
\hspace{.2cm} \mbox{for any $P$}.
$$
The goal of this section is to prove Theorem \ref{mainthmbis:Dmod ss principal ideal}, which states that any object of $\Dmod(\LSG)^\ss$ is \virg{divisible} by $\St_G$. In other words:

\begin{thm} \label{thm:Dmod-ss generated by St}
The full subcategory $\Dmod(\LSG)^\ss \subseteq \Dmod(\LSG)$ is a principal monoidal ideal generated by $\St_G \in \Dmod(\LSG)$.
\end{thm}

\begin{rem}

Since $\Dmod(\LSG)^\ss$ is evidently cocomplete, it follows that the same holds for the essential image of $\St_G \sotimes -: \Dmod(\LSG) \to \Dmod(\LSG)$. This will be clear from Theorem \ref{thm:Div functor} below, which identifies the \virg{quotient} by $\St_G$ explicitly.

\end{rem}

\sssec{}

For any $P \in \Par$, consider the adjunction $(\CT^\fD_{P^-,!}, \Eis^\fD_{P^-,*})$ introduced in Section \ref{ssec:Braden}. As $P$ varies in $\Par'$, the units of these adjunctions allow to form the limit functor
$$
\lim_{P \in (\Par')^\op}
\Eis^\fD_{P^-,*}
\circ
\CT^\fD_{P^-,!},
$$
as well as a natural arrow $\epsilon: \id_{\Dmod(\LSG)} \to \lim_{P \in (\Par')^\op}
\Eis^\fD_{P^-,*}
\circ
\CT^\fD_{P^-,!}$.
We then set
$$
\Div_G := 
\ker
\Bigt{
\id 
\xto{ \; \;
\epsilon \;\; }
\lim_{P \in (\Par')^\op}
\Eis^\fD_{P^-,*}
\CT^\fD_{P^-,!}
}:
\Dmod(\LSG)^\ss \longto \Dmod(\LSG).
$$
Note that $\Div_G$ is defined, with the same formula, on the entire $\Dmod(\LSG)$. However, we are only interested in it as a functor out of $\Dmod(\LSG)^\ss$.

\begin{thm} [Divisibility by the Steinberg D-module] \label{thm:Div functor}

The functor
$$
\Div_G:
\Dmod(\LSG)^\ss \longto \Dmod(\LSG)
$$
defined above is a section of 
$\St_G \sotimes -: \Dmod(\LSG) \to \Dmod(\LSG)^\ss$.
\end{thm}

\begin{proof}

First, let us reiterate our usual notational convention: since we are only dealing with D-modules, we omit the decoration \virg{$\dR$} on pullback and pushforward functors.

\sssec*{Step 0}

The theorem states that any $\F \in \Dmod(\LSG)^\ss$ is isomorphic to $\St_G \sotimes \Div_G(\F)$. To prove this, it suffices to exhibit an isomorphism
\begin{equation}\label{eqn:main-formula-for-Dmod-princ}
\St_G
\sotimes
\lim_{P \in (\Par')^\op}
\Eis^\fD_{P^-,*}
\CT^\fD_{P^-,!}(\F)
\simeq
\Bigt{ \uscolim{P \in \Par'} \;
(\p_P)_!(\omega_{\LSP}) } [1]
\sotimes
\F
\end{equation}
that intertwines $\St_G \sotimes \epsilon$ with the arrow induced by the structure map $\St_G \to \Bigt{ \uscolim{P \in \Par'} \;
(\p_P)_!(\omega_{\LSP}) } [1]$.

\sssec*{Step 1}

Theorem \ref{mainthm: StG vs StM}, combined with the projection formula, yields
$$
\St_G
\sotimes
\Eis^\fD_{P^-,*}( - )
\simeq
(\i_{MG})_*(\St_M \sotimes -)[\rk G - \rk M],
$$
where we denote by $\i_{MG}: \LSM \to \LSG$ the induction map.
It follows that
$$
\St_G
\sotimes
\Eis^\fD_{P^-,*}
\CT^\fD_{P^-,!}(\F)
\simeq
(\i_{MG})_*
\Bigt{ 
\St_M \sotimes \CT^\fD_{P^-,!} (\F)
}[\rk G - \rk M].
$$

\sssec*{Step 2}

By the second adjunction $\CT^\fD_{P^-,!} \simeq \CT^\fD_{P,*}$, the latter is isomorphic to 
$$
(\i_{MG})_*
\Bigt{ 
\St_M \sotimes \CT^\fD_{P,*} (\F)
}[\rk G - \rk M],
$$
and further, by the definition of $\St_M$, to
\begin{equation} \label{eqn:super-long}
\coker
\Bigt{
\uscolim{Q \subsetneq P}
\;
(\i_{MG})_*
\bigt{ 
(\p_{Q \cap M \to M})_*(\omega_{\LS_{Q \cap M}}) \sotimes \CT^\fD_{P,*} (\F)
}
\to
(\i_{MG})_* \CT^\fD_{P,*} (\F)
}
[\rk G - \rk M].
\end{equation}
In the next two steps, we use the assumption that $\F \in \Dmod(\LSG)^\ss$ to simplify this expression.

\sssec*{Step 3}

We have:
$$
(\i_{MG})_* 
\Bigt{ 
\CT^\fD_{P,*} (\F)
}
\simeq
(\p_P)_* (\i_P)_*
\Bigt{ 
(\q_P)_*
(\p_P)^! (\F)
}
\simeq
(\p_P)_* 
\circ
\Bigt{ 
(\i_P)_*
(\q_P)_*
}
\circ
(\p_P)^! (\F).
$$
Now recall that $\F \in \Dmod(\LSG)^\ss$, so that $(\p_P)^!(\F) \simeq (\i_P)_*(\CT_{P,*}^\fD(\F))$. It follows that the monad $(\i_P)_*(\q_P)_*$ acts as the identity on $(\p_P)^! (\F)$. We conclude that 
$$
(\i_{MG})_* 
\Bigt{ 
\CT^\fD_{P,*} (\F)
}
\simeq
(\p_P)_* 
\circ
(\p_P)^! (\F)
\simeq
(\p_P)_* (\omega_{\LSP}) \sotimes \F.
$$

\sssec*{Step 4}

A similar argument yields
$$
(\i_{MG})_*
\bigt{ 
(\p_{Q \cap M \to M})_*(\omega_{\LS_{Q \cap M}}) \sotimes \CT^\fD_{P,*} (\F)
}
\simeq
(\p_Q)_* (\omega_{\LS_Q}) \sotimes \F.
$$
Combining this and the above step, we obtain that \eqref{eqn:super-long} simplifies as
\begin{equation} \label{eqn:super-long-simplified}
\nonumber
\coker
\left(
\left(
\uscolim{Q \subsetneq P}
\;
(\p_Q)_* (\omega_{\LS_Q}) \sotimes \F
\right)
\to
(\p_P)_* (\omega_{\LSP}) \sotimes \F
\right)
[\rk G - \rk M].
\end{equation}

\sssec*{Step 5}

Unwinding the constructions, we obtain that the LHS of \eqref{eqn:main-formula-for-Dmod-princ} is isomorphic to the tensor product of $\F$ with the object
$$
\V := 
\lim_{P \in (\Par')^\op}
\coker
\Bigt{
\uscolim {Q \subsetneq P} 
(\p_Q)_! (\omega_{\LS_Q})
\to
(\p_P)_! (\omega_{\LS_P})
}
[ \rk G - \rk M].
$$
Thus, to obtain an isomorphism as in \eqref{eqn:main-formula-for-Dmod-princ}, it suffices to exhibit an isomorphism
\begin{equation} \label{eqn:topo}
\V
\simeq
\Bigt{ \uscolim{P \in \Par'} \;
(\p_P)_!(\omega_{\LSP}) } [1].
\end{equation}
The construction of this isomorphism (to be performed in the next two steps) will be compatible with the functor $\St \sotimes \epsilon$, as requested in Step $0$; we leave the details to the reader. 

\sssec*{Step 6}

Denote by $\phi: \Par \to \Dmod(\LSG)$ the functor $P \squigto (\p_P)_!(\omega_{\LSP})$. In the spirit of Lemma \ref{lem:colim of poset of parts }, consider the poset $\P'(I \sqcup \infty)$ of proper subsets $I \sqcup \infty$. Here, $I$ is the set of nodes of the Dynkin diagram of $G$ and $\infty$ is an extra node. For any $P \in \Par$, corresponding to the subset $J_P \subseteq I$, we define
$$
\wh \phi_P:  \P'(I \sqcup \infty) \longto \Dmod(\LSG)
$$
as
$$
\wh \phi_P(J) =
\begin{cases}
\phi(J) & \mbox{ if $J \subseteq J_P$;} \\
0  & \mbox{ othewise}.
\end{cases}
$$

\sssec*{Step 7}

Note, in passing, that $\colim \, \wh \phi_G \simeq \St_G$ by definition.
Similarly, by Lemma \ref{lem:colim of poset of parts }, we obtain that
$$
\colim \, \wh \phi_P
\simeq
\coker
\Bigt{
\uscolim {Q \subsetneq P} 
(\p_Q)_! (\omega_{\LS_Q})
\to
(\p_P)_! (\omega_{\LS_P})
}
[ \rk G - \rk M].
$$
This allows to rewrite $\V$ simply as
$$
\V \simeq
\lim_{P \in (\Par')^\op}
\colim \, \wh\phi_P.
$$

\sssec*{Step 8}

In a stable $\infty$-category, finite limits commute with finite colimits, whence
$$
\V \simeq
\uscolim{J \in \P'(I \sqcup \infty)}
\lim_{P \in (\Par')^\op}
\wh\phi_P(J).
$$
It is easy to see that
$$
\lim_{P \in (\Par')^\op}
\wh\phi_P(J)
=
\begin{cases}
\phi(J) & \mbox{ if $J \subsetneq I$;} \\
0  & \mbox{ otherwise}.
\end{cases}
$$
From this, it is clear that 
$$
\V \simeq \uscolim{J \subsetneq I}  \,\phi(J) [1],
$$ 
as desired.
\end{proof}

\ssec{The quasi-coherent case} \label{ssec:qcoh case}

Finally, let us prove that $\QCoh(\LSG)^\ss$ is a principal monoidal ideal generated by $\ul\St_G$.

\sssec{}

By definition, $\QCoh(\LSG)^\ss$ is the cocompletion of the essential image of the functor
$$
\QCoh(\LSG)
\usotimes{\Dmod(\LSG)}
\Dmod(\LSG)^\ss
\longto
\QCoh(\LSG).
$$
Thus, Theorem \ref{thm:Dmod-ss generated by St} implies that $\QCoh(\LSG)^\ss$ is generated under colimits by the essential image of the functor $\ul\St_G \otimes -: \QCoh(\LSG) \to \QCoh(\LSG)$.
It remains to show that any such colimit can be rewritten as a single tensor product $\F \otimes \ul\St$. This can be done as follows, with the help of Theorem \ref{mainthm:St-fully faithful on CohN}.

\begin{lem} \label{lem:ulSt a single tensor}
Suppose that $\F = \colim \psi \in \QCoh(\LSG)$ is the colimit of a diagram
$$
\psi:
A \to \QCoh(\LSG), 
\hspace{.4cm}
a \squigto \F_a,
$$
with each $\F_a$ belonging to the essential image of $\ul\St_G \otimes -: \QCoh(\LSG) \to \QCoh(\LSG)$. Then $\F \simeq \ul\St_G \otimes \G$ for some $\G \in \QCoh(\LSG)$.
\end{lem}

\begin{proof}
Consider the functor 
\begin{equation} \label{eqn:pippo}
\ul\St_G \otimes - : \Perf(\LSG) \to \QCoh(\LSG)
\end{equation}
and let $\C$ be the cocompletion of its essential image. 
We claim that the obvious fully faithful embedding $\C \subseteq \QCoh(\LSG)^\ss$ is an equivalence. This follows from the fact that $\Ind(\Perf(\LSG)) \simeq \QCoh(\LSG)$, i.e., any object of $\QCoh(\LSG)$ is a filtered colimit of perfect objects. 

Hence, it suffices to prove the lemma under the assumption that each $\F_a$ belongs to the essential image of \eqref{eqn:pippo}. Then, for any $a \in A$, there exist a \emph{perfect} object $\P_a$ and an isomorphism $\ul\St_G \otimes \P_a \simeq \F_a$.
The fully faithfulness result of Theorem \ref{mainthm:St-fully faithful on CohN} implies that the given functor $\psi$ determines a functor
$$
\phi:
A \to \Perf(\LSG) \subseteq \QCoh(\LSG), 
\hspace{.4cm}
a \squigto \P_a
$$
and a natural equivalence $\psi \simeq \ul\St_G \otimes \phi$. Taking colimits (and using the fact that colimits commute with tensor products), we finally obtain 
$
\F 
=
\colim \psi 
\simeq 
\colim (\ul\St_G \otimes \phi) 
\simeq \ul\St_G \otimes \colim(\phi),
$ as desired.
\end{proof}

\end{document}